\documentclass[11pt]{amsart}
\usepackage{amssymb,amsbsy,upref,epsf,MnSymbol}
\usepackage{amsmath,amssymb,amsthm}
\usepackage{graphicx,mathtools,mathrsfs,wrapfig}
\usepackage[margin=1in]{geometry}
\usepackage{fancyhdr}
\usepackage{enumerate}

\newtheorem{theorem}{Theorem}[section]
\newtheorem{proposition}[theorem]{Proposition}
\newtheorem{lemma}[theorem]{Lemma}
\theoremstyle{remark}
\newtheorem*{remark}{Remark}
\theoremstyle{definition}
\newtheorem{definition}{Definition}
\newcommand{\R}{\mathbb{R}}
\newcommand{\N}{\mathbb{N}}
\newcommand{\Z}{\mathbb{Z}}

\newcommand{\C}{\mathbb{C}}

\newcommand{\Four}{\mathcal{F}}
\newcommand{\T}{\mathbb{T}}
\newcommand{\E}{\mathcal{E}}

\newcommand{\eps}{\varepsilon}

\newcommand{\chevron}[1]{\langle #1 \rangle}
\newcommand{\norm}[1]{\left\lVert#1\right\rVert}
\newcommand{\paren}[1]{\left( #1 \right)}
\newcommand{\bracket}[1]{\left[ #1 \right]}
\newcommand{\abs}[1]{\left\lvert #1 \right\rvert}

\DeclareMathOperator{\supp}{supp}

\DeclareMathOperator{\lspan}{span}

\DeclareMathOperator{\BMO}{BMO}
\newcommand{\del}{\partial}

\newcommand{\grad}{\nabla}
\newcommand{\ddt}{\frac{d}{dt}}
\renewcommand{\div}{\operatorname{div}}
\newcommand{\Laplace}{\Delta}

\newcommand{\Lip}{\text{Lip}}
\renewcommand{\Re}{\operatorname{Re}}
\newcommand{\ddz}{\frac{d}{dz}}

\newcommand{\ith}{^\mathrm{th}}
\newcommand{\n}{^{-1}}

\newcommand{\indic}[1]{\chi_{\{#1\}}}

\newcommand{\eigen}[1]{\eta_{#1}} 
\newcommand{\Ctest}{C_c^\infty}

\newcommand{\ulow}{u_\ell}
\newcommand{\uhigh}{u_h}
\newcommand{\ulowth}[1]{\ulow^{#1}}
\newcommand{\uhighth}[1]{\uhigh^{#1}}

\newcommand{\Gammal}{\Gamma_\ell}

\newcommand{\HD}{\mathcal{H}}
\newcommand{\HDint}[2]{\int \abs{\Lambda^{#1} #2}^2}

\newcommand{\Ccalib}{\kappa}
\newcommand{\Cgamma}{C_\mathit{pth}}
\newcommand{\Comega}{C_\mathit{dmn}}
\newcommand{\Csuit}{C^\ast}

\newcommand{\Rom}[1]{\MakeUppercase{\romannumeral #1}}

\newcounter{step_count}[section]
\newcommand{\step}[1]{\stepcounter{step_count} \smallskip \noindent{\textbf{Step \arabic{step_count}:} #1}}

\title[Boundary Regularity for SQG]{Holder regularity up to the boundary for critical SQG on bounded domains}

\author[Stokols]{Logan F. Stokols} 
\address[L. F. Stokols]{\newline Department of Mathematics, \newline The University of Texas at Austin, Austin, TX 78712, USA}
\email{lstokols@math.utexas.edu}

\author[Vasseur]{Alexis F. Vasseur}
\address[A. F. Vasseur]{\newline Department of Mathematics, \newline The University of Texas at Austin, Austin, TX 78712, USA}
\email{vasseur@math.utexas.edu}

\date{\today}


\thanks{\textbf{Acknowledgment.} This work was partially supported by the NSF Grant DMS 1614918 and the NSF RTG Grant DMS 1840314. 
We would also like to thank the referee for suggesting several improvements to the paper, especially the construction of global suitable solutions.}

\begin{document}

\begin{abstract}
We consider the dissipative SQG equation in bounded domains, first introduced by Constantin and Ignatova in 2016. 
We show global Holder regularity up to the boundary of the solution, with a method based on the De Giorgi techniques. The boundary introduces several difficulties. In particular, the Dirichlet Laplacian is not translation invariant near the boundary, which leads to complications involving the Riesz transform.


\end{abstract}

\maketitle \centerline{\date}

\tableofcontents

\section{Preliminaries} \label{preliminaries}
The surface quasigeostrophic equation (SQG) is a special case of the quasi-geostrophic system (QG) with uniform potential vorticity. The QG model is used extensively in meteorology and oceanography (e.g. Charney \cite{Ch}). These models are described in Pedlosky \cite{Pe}. The SQG model was popularized by Constantin, Majda and Tabak in \cite{CoMaTa}, due to its similarities with the Euler and Navier-Stokes  equation. They proposed it as a toy model for the study of 3D Fluid equations (see also Held, Garner, Pierrehumbert, and Swanson \cite{GaHePiSw}). 

\vskip0.3cm
 We consider in this paper  critical SQG on a bounded domain. We will focus on the following model, which was introduced by Constantin and Ignatova in \cite{CoIg.fraclap} and \cite{CoIg.sqg}.
 Consider  $\Omega$ a connected bounded domain in $\R^2$ with $C^{2,\beta}$ boundary for some $\beta \in (0,1)$, and the Laplacian with homogeneous Dirichlet boundary conditions $-\Laplace_D$.  If $(\eigen{k})_{k \in \N}$ is the sequence of $L^2$-normalized eigenfunctions of $-\Laplace_D$ with corresponding eigenvalues $\lambda_k$ listed in non-decreasing order, define
\[ \Lambda f := \sum_{k=0}^\infty \sqrt{\lambda_k} \chevron{f,\eigen{k}}_{L^2(\Omega)} \eigen{k}. \]
The critical SQG problem on $\Omega$ with initial data $\theta_0 \in L^2(\Omega)$ is
\begin{equation} \label{eq:main nonlinear} \begin{cases}
\del_t \theta + u\cdot \grad \theta + \Lambda \theta = 0 & (0,T) \times \Omega,
\\ u = \grad^\perp \Lambda\n \theta & [0,T] \times \Omega,
\\ \theta = \theta_0 & \{0\} \times \Omega.
\end{cases} \end{equation}

In the model, the dissipation $\Lambda =\paren{-\Delta_D}^{1/2}$  is due to the  Ekman pumping, while the nonlinear velocity $u$ comes from the geostrophic and hydrostatic balance (see \cite{Pe}).
\vskip.3cm

The main result of this paper is the following:
\begin{theorem} \label{thm:main continuity}
There exists a universal constant $C_1 > 0$ such that the following holds:

For any $\Omega \subseteq \R^2$ open and bounded with $C^{2,\beta}$ boundary, $\beta \in (0,1)$, there exists for any $S > 0$ a constant $C_S > 0$ (depending also on $\Omega$), and for any $k > 0$ a constant $\alpha_{k,S} \in (0,1)$ (depending also on $\Omega$) so the following holds:
\vskip.3cm

For any $\theta_0 \in L^2(\Omega)$ there exists a global-in-time weak solution $\theta \in L^\infty([0,\infty);L^2(\Omega)) \cap L^2([0,\infty);\HD^{1/2})$ to \eqref{eq:main nonlinear} verifying $\theta(t,x) = 0$ on $(0,\infty)\times \del\Omega$ and $\lim_{t \to 0} \theta(t,\cdot) = \theta_0$ in the $L^2$-weak sense.  For $k \geq \norm{\theta_0}_{L^2(\Omega)}$ and for every $S > 0$
\[ \theta \in C^{\alpha_{k,S}}([S,\infty)\times\bar{\Omega}) \]
where $\bar{\Omega}$ denotes the closure of $\Omega$.  

Moreover,
\[ \norm{\theta}_{L^\infty([S,\infty)\times\bar{\Omega})} \leq \frac{C_1}{S} \norm{\theta_0}_{L^2(\Omega)} \]
and
\[ \norm{\theta}_{C^{\alpha_{k,S}}([S,\infty)\times\bar{\Omega})} \leq C_S \norm{\theta_0}_{L^2(\Omega)}. \]

\end{theorem}

This model was first thoroughly studied in the cases without boundaries (either $\R^2$ or the torus $\T^2$).   
Global weak solutions were first constructed in Resnick \cite{Re}. Global regularity was first shown with small initial values by Constantin, Cordoba, and Wu \cite{CoCoWu}, or extra $C^\alpha$ regularity on the velocity in Constantin and Wu \cite{CoWu} and Dong and Pavlovi\'{c} \cite{DoPa}. In \cite{KiNaVo}, Kiselev, Nazarov and Volberg showed the 
propagation of $C^\infty$ regularity. The global $C^\infty$ regularity for any $L^2$ initial values was first proved in  \cite{CaVa.sqg} (see also Kiselev and Nazarov \cite{KiNa.variations} and Constantin and Vicol \cite{CoVi}).   

\vskip0.3cm

In the presence of boundaries, there are several  distinct ways to define SQG. This can be attributed  to alternative generalizations of the fractional Laplacian.  Kriventsov \cite{Kr} considered a two-phase problem which satisfies critical SQG only in part of the domain, and was able to prove H\"{o}lder regularity in the time-independent case.  This problem, intended to model air currents over a region containing both land and water, contains a half-Laplacian and a Riesz transform defined, not spectrally, but in terms of extension.  In \cite{NoVa.bounded}, the authors consider the Euler-Coriolis-Boussinesq model and derive the full 3D inviscid quasigeostrophic system in an impermeable cylinder (see also \cite{NoVa.solutions} for the construction of small time smooth solutions to the model).  They obtain natural boundary conditions for SQG distinct from the homogeneous conditions introduced in \cite{CoIg.fraclap}, \cite{CoIg.sqg} and described above. 
However, due to the complexity of the model described in \cite{NoVa.bounded}, we focus in this paper only on the homogenous case.

\vskip0.3cm
Existence of weak solutions for \eqref{eq:main nonlinear}  is proven in \cite{CoIg.fraclap}, and local existence and uniqueness for strong solutions with sufficiently smooth initial data is proven by Constantin and Nguyen in \cite{CoNg.strong} (see also Constantin and Nguyen \cite{CoNg} and Constantin, Ignatova, and Nguyen \cite{CoIgNg} for the inviscid case). The interior regularity of solutions  is proven in \cite{CoIg.sqg} (together with propagation of $L^\infty$ bounds).  The method of proof for interior regularity uses  nonlinear maximum principles, introduced by Constantin and Vicol \cite{CoVi}.   However, the bounds obtained in \cite{CoIg.sqg} blow up near the boundary and do not provide global regularity.
  In \cite{CoIg.sqg} Remark 1, questions about   global regularity are suggested as  open problems.  Both the $C^\alpha(\bar{\Omega})$ regularity,  and bootstrapping to  $C^\infty(\bar{\Omega})$ regularity, are indentified as interesting problems.  
  Our  result answers the first question, by  showing  that solutions $\theta$ to (\ref{eq:main nonlinear}) are globally H\"{o}lder continuous.
  Bootstrapping to $C^\infty$ involves different techniques, and will be studied in a forthcoming work \cite{StVa.higher}.  
\vskip0.3cm

Our proof is based on  the De Giorgi method pioneered by De Giorgi in \cite{DG}.  The method was applied to the SQG problem first in \cite{CaVa.sqg}.  The method is powerful for showing $C^\alpha$ regularity of elliptic- and parabolic-type equations. It has been applied in a variety of situations for non-local problems, such as the fractional heat equation in \cite{CaChVa.nio}, the time-fractional case in \cite{AlCaVa}, the 3D Quasigeostrophic problem in \cite{NoVa.qg}, or the kinetic setting by Imbert and Silvestre \cite{ImSi} or in \cite{St}.  The method has also been applied in more exotic, non-elliptic situations such as Hamilton-Jacobi equations (see \cite{ChVa}, \cite{StVa.hamjac}).  

The De Giorgi method involves rescaling our equation by zooming in iteratively, and applying regularity results at each scale.  Therefore it is important that certain results be proven independently of the domain $\Omega$.  
The particular dependence on $\Omega$ will be made clear in each lemma of this paper.  
As a general overview, in the proof of Theorem~\ref{thm:main continuity} we will apply the results of Sections~\ref{sec:suitable} and \ref{sec:littlewood paley} only on a single fixed domain, while the results of Sections~\ref{sec:de giorgi} and \ref{sec:harnack} must be applied at each level of zoom with a different rescaled domain each time.  

\vskip0.3cm

The first broad  idea of our proof consists in decoupling the  velocity $u$ from  $\theta$ to work on  a linear equation, and prove alternating regularity results for $\theta$ and $u$ independently.  We can show that $\theta$ is in $L^\infty$ without any assumption on $u$ (see Section~\ref{sec:suitable}).  
Using that $L^\infty$ bound, we will need to obtain scaling invariant controls on the drift   $u = \grad \Lambda^{-1} \theta$.  By scaling invariant, we mean that the bound, once proven on $\Omega$ fixed, will remain true of the scaled function $u(\eps \, \cdot, \eps \, \cdot)$ for all $\eps$.  
Unfortunately, although the Riesz transform is bounded from $L^p$ to $L^p$ for all $p$ finite, it is not bounded for $p = \infty$.  
The usual technique, therefore, is to consider $\BMO$ (as in \cite{CaVa.sqg} and \cite{NoVa.qg}), but in the case of bounded domains the Riesz transform is not known to be bounded in this space either.  Our solution is to use extensions of the Littlewood-Paley theory to bounded domains. 

The adaptation of Fourier analysis and Littlewood-Paley theory to Schrodinger operators is a well-studied subject (e.g. Zheng \cite{Zh}, Benedetto and Zheng \cite{BeZh}).  As an application of this theory, Iwabuchi, Matsuyama, and Taniguchi \cite{IMT.besov}, \cite{IMT.schrodinger}, and Bui, Duong, and Yang \cite{BuDuYa} have considered operators defined on open subsets of $\R^n$, which includes as a special case the operator $-\Laplace_D$ (a Schrodinger operator with zero potential).  In particular, in \cite{IMT.bilinear},  Iwabuchi, Matsuyama, and Taniguchi derive many important results, including the Bernstein inequalities, for Besov spaces adapted to the operator $-\Laplace_D$ on bounded open subsets of $\R^n$ with smooth boundary.  This theory turns out to greatly improve our understanding of the Riesz transform $\grad \Lambda^{-1}$ on bounded domains.

Using the results of \cite{IMT.bilinear}, we will be able to show that the Riesz transform of an $L^\infty$ function whose Fourier decomposition $f = \sum f_k \eigen{k}$ is supported on high frequencies $k > N$ will be bounded in the weak sobolev space $W^{-1/4,\infty}$, and the Riesz transform of an $L^\infty$ function whose Fourier decomposition is supported on low frequencies $k<N$ will have bounded Lipschitz constant.  The cutoff $N$ for dividing high frequencies from low frequencies must depend however on the size of the domain $\Omega$.  In the case of $\R^2$, where $\grad$  and $\Lambda^{-1}$ commute, this is equivalent to the observation that the Riesz transform is bounded from $L^\infty$ to the Besov space $B_{\infty,\infty}^0$.  In the case of bounded domains, the argument must be more subtle.  We must decompose $\theta$ into its Littlewood-Paley projections, individually bound the Riesz transform of each projection in multiple spaces, and then recombine these infinitely-many functions into a low-frequency collection and a high-frequency collection depending on the scale of oscillation we are trying to detect (see Section~\ref{sec:littlewood paley} and Lemma~\ref{thm:calibration is good}).  

We make this notion precise with the following definition:

\begin{definition}[Calibrated sequence] \label{def:calibrated}
Let $\Omega\subseteq \R^2$ be any bounded open set and $0<T\in\R$.  We call a function $u\in L^2([0,T]\times\Omega)$ \textbf{calibrated} if it can be decomposed as the sum of a calibrated sequence 
\[ u = \sum_{j \in \Z} u_j \]
with each $u_j \in L^2([0,T]\times\Omega)$ and the infinite sum converging in the sense of $L^2$.  

We call a sequence $(u_j)_{j\in\Z}$ \textbf{calibrated} for a constant $\Ccalib$ and a center $N$ if each term of the sequence satisfies the following bounds.  
\begin{align*}
\norm{u_j}_{L^\infty([0,T]\times\Omega)} &\leq \Ccalib, \\
\norm{\grad u_j}_{L^\infty([0,T]\times\Omega)} &\leq 2^{j} 2^{-N} \Ccalib, \\
\norm{\Lambda^{-1/4} u_j}_{L^\infty([0,T]\times\Omega)} &\leq 2^{-j/4} 2^{N/4} \Ccalib.  
\end{align*} 

\end{definition}
\vskip.3cm

In Section \ref{sec:holder} we will show that a calibrated velocity remains calibrated at all scales (specifically, with fixed constant $\Ccalib$ but a changing center $N$).  Therefore we can consider, for any domain $\Omega$ and time $T$, the system of linear equations
\begin{equation} \label{eq:main linear} \begin{cases}
\del_t \theta + u \cdot \grad \theta + \Lambda \theta = 0, & [-T,0]\times\Omega \\
\div u = 0 & [-T,0]\times\Omega.
\end{cases} \end{equation}

In Section \ref{sec:suitable} we show that solutions to \eqref{eq:main nonlinear} with $L^2$ initial data exist and regularize instantly into $L^\infty$, and in Section~\ref{sec:littlewood paley} we show that the Riesz transform of $L^\infty$ data is calibrated.  Then in Sections \ref{sec:de giorgi} and \ref{sec:harnack} we will show that solutions to \eqref{eq:main linear} with calibrated velocity have decreasing oscillation between scales.  By iteratively applying this oscillation lemma and scaling our equation, we show in Section~\ref{sec:holder} that $\theta$ is H\"{o}lder continuous.  
\vskip.3cm

The low-freqency component of a calibrated velocity $u$ will be uniformly Lipschitz, which means it is only bounded up to a constant.  This is similar to the case of BMO velocity functions in \cite{CaVa.sqg} and \cite{NoVa.qg}, which by the John-Nirenberg inequality are also bounded up to a constant.  As in these cases, we consider a moving reference frame, denoted $\Gamma:[0,T] \to \R^2$, in which our velocity is shifted by a constant, making the low-frequency component of $u$ bounded. There are two differences between our implementation of this technique and the implementation in \cite{CaVa.sqg} and \cite{NoVa.qg}: firstly, we subtract off the value of the low-frequency part of $u$ at a point, rather than subtracting off the average of $u$ on a ball.  Secondly, rather than applying the standard De Giorgi argument to $\tilde{\theta}(t,x) := \theta(t,x+\Gamma(t))$, we must reformulate the De Giorgi argument to ``follow'' the path $\Gamma(t)$ explicitly.  This is a purely notational difference, but it is necessary because otherwise $\Omega$ would be time-dependent.  
%

At each scale, there will be a natural Lagrangian path $\Gammal$ corresponding to the low-frequency part of $u$.  However, the low-frequency part of $u$ changes non-trivially as we zoom, so $\Gammal$ will be different at each scale.  Throughout Sections~\ref{sec:de giorgi} and \ref{sec:harnack}, we will use $\Gammal$ to denote the ``current'' Lagrangian path and $\Gamma$ to denote the Lagrangian path at the previous scale. In the proof of Theorem~\ref{thm:main continuity} in Section~\ref{sec:holder}, these are denoted $\Gamma_k(t)$ and $\eps\n \Gamma_{k-1}(\eps t)$ respectively.  In Lemmas~\ref{thm:energy inequality}, \ref{thm:DG1}, \ref{thm:DG2} and \ref{thm:oscillation general}, we will make assumptions about $\theta$ which are centered on $x \approx \Gamma(t)$ and obtain conclusions which are similarly centered on $x \approx \Gamma(t)$, conditioned on $\gamma := \Gammal-\Gamma$ being small in Lipschitz norm.  Finally in Lemma~\ref{thm:oscillation shifted}, we will show that, given bounds on $\theta$ for $x \approx \Gamma(t)$, we can bound $\theta$ for $x \approx \Gammal(t)$ for $t$ sufficiently small, again conditioned on $\gamma := \Gammal-\Gamma$ being small in Lipschitz norm.  Controlling $\gamma$ amounts to controlling the change in $\Gammal$ between consecutive scales, which is much easier to obtain than scale-independent bounds on $\Gammal$.  
%
%
\vskip.3cm

Previous applications of the De Giorgi method to non-local equations such as \eqref{eq:main linear} generally make extensive use of either an extension representation (c.f. \cite{CaVa.sqg}) or a singular integral representation (c.f. \cite{NoVa.qg}).  In this paper, we  use the singular integral representation for the Dirichlet fractional Laplacian  derived by Caffarelli and Stinga \cite{CaSt}. It is based on the results of Stinga and Torrea \cite{StTo} which generalize the extension representation of Caffarelli and Silvestre \cite{CaSi}. 
This theory is pivotal in translating the existing non-local De Giorgi techniques to the problem at hand (see Section~\ref{sec:lambda}).  
\vskip.3cm

In order to apply De Giorgi's method to weak solutions of \eqref{eq:main linear}, we will need to assume a certain a priori estimate which holds, in particular, for $L^2(H_0^1)$ weak solutions.  However, such solutions are only known to exist for short time and for $H^2$ initial data, as shown by Constantin and Nguyen in \cite{CoNg.strong}.  We call weak solutions in $L^2(\HD^{1/2})$ which happen to verify this a priori estimate ``suitable solutions,'' by analogy to suitable solutions to Navier-Stokes as in \cite{CaKoNi}.  We give the formal definition of suitable solutions in Section~\ref{sec:suitable}, where we also construct global-in-time suitable solutions using the vanishing viscosity method.  Compared to \cite{CoIg.fraclap}, our solutions verify a full family of localized energy inequalities which allow us to apply the De Giorgi method.  
\vskip.3cm

The Paper is organized as follows. Section \ref{sec:lambda} is dedicated to basic properties of the operator $\Lambda$ and the corresponding Sobolev spaces $\HD^s$.  In Section~\ref{sec:suitable} we construct weak solutions which verify the suitability conditions.  
In Section~\ref{sec:littlewood paley} we prove that the Riesz transform of the $L^\infty$ function $\theta$ is callibrated.  Section~\ref{sec:de giorgi} contains the De Giorgi Lemmas.  Section~\ref{sec:harnack} is dedicated to the local decrease in oscillation through an analog of the Harnack inequality.  Finally in Section~\ref{sec:holder} we prove the main theorem, Theorem~\ref{thm:main continuity}.  In the Appendix \ref{sec:technical} we prove a few technical lemmas which are needed in the main paper.  
\vskip.3cm

\textbf{Notation.} Throughout the paper, we will use the following notations.
By $\eigen{k}$ and $\lambda_k$ we mean the eigenfunctions and eigenvalues of $-\Laplace_D$, with $\lambda_0 \leq \lambda_1 \leq \ldots$ and $\norm{\eigen{k}}_2 = 1$ for all $k$.  
If $f = \sum_k f_k \eigen{k}$ then
\begin{align*} 
\norm{f}_{\HD^s} &:= \paren{\sum_k \lambda_k^{s} f_k^2}^{1/2} 
\\ &= \HDint{s}{f}. 
\end{align*}
We suppress the dependence on $\Omega$, though in fact $\Lambda$, $\lambda_k$, and $\eigen{k}$ are defined in terms of the domain $\Omega$.  The relevant domain will be clear from context.  The norm on $\HD^s$ is in fact a norm, not a seminorm, since $\norm{f}_{L^2(\Omega)} \leq \lambda_0^{-s/2} \norm{f}_{\HD^s}$.  

For a set $A$ and a function $f:A \to \R$, denote
\begin{align*}
\bracket{f}_{\alpha;A} &:= \sup_{x,y \in A, x \neq y} \frac{|f(x)-f(y)|}{|x-y|^\alpha},  &\alpha \in (0,1], \\
\norm{f}_{C^\alpha(A)} &:= \norm{f}_{L^\infty(A)} + \bracket{f}_{\alpha;A}, & \alpha \in (0,1], \\
\norm{f}_{C^{k,\alpha}(A)} &:= \sum_{n=0}^k \norm{D^n f}_{L^\infty(A)} + \bracket{D^k f}_{\alpha;A}, & \alpha \in (0,1], k \in \N.
\end{align*}
When the domain $A$ is ommited, the relevant spatial domain $\Omega$ is implied.  

We will use the notation $(x)_+ := \max(0,x)$.  When the parentheses are ommited, the subscript~$+$ is merely a label.  

Throughout this paper, if an integral sign is written $\int$ without a specified domain, the domain is implied to be $\Omega$, with $\Omega$ defined in context.  

For any vector $v = (v_1,v_2)$, by $v^\perp$ we mean $(-v_2,v_1)$. By $\grad^\perp$ we mean $(-\del_y, \del_x)$.  

In the remainder of this paper, the differential operator $D^2$ refers to the Hessian in space, excluding time derivatives.  The function space $\Ctest$ consists of smooth functions with compact support.  By abuse of notation, if $\Gamma: [a,b] \to \R^2$, we write $[a,b] \times B_R(\Gamma)$ to denote $\{(t,x) \in [a,b]\times\R^2: |x-\Gamma(t)| \leq R\}$.  


The symbol $C$ represents a constant which may change value each time it is written.  

\vskip1cm
\section{Properties of the Fractional Dirichlet Laplacian} \label{sec:lambda}

In this section we will investigate the basic properties of the operator $\Lambda$ and the space $\HD^s$ on a general domain $\Omega$.  

We begin by stating a result of \cite{CaSt} which gives us a singular integral representation of the $\HD^s$ norm.  
\begin{proposition}[Caffarelli-Stinga Representation] \label{thm:Caff Stinga representation}
Let $s \in (0,1)$ and $f,g \in \HD^{s}$ on a bounded $C^{2,\beta}$ domain $\Omega \subseteq \R^2$.  Then
\begin{equation} \label{Caff Stinga representation}
\int_\Omega \Lambda^s f \Lambda^s g \,dx = \iint_{\Omega^2} [f(x)-f(y)][g(x)-g(y)] K_{2s}(x,y) \,dxdy + \int_{\Omega} f(x) g(x) B_{2s}(x) \,dx 
\end{equation}
for kernels $K_{2s}$ and $B_{2s}$ which depend on the parameter $s$ and the domain $\Omega$.  
\vskip.3cm

There exists a constant $C = C(s)$ independent of $\Omega$ such that
\[ 0 \leq K_{2s}(x,y) \leq \frac{C(s)}{|x-y|^{2+2s}} \]
for all $x \neq y \in \Omega$ and
\[ 0 \leq B_{2s}(x) \]
for all $x \in \Omega$.  

Moreover, for any $s,t \in (0,2)$ there exists a constant $c = c(s,t,\Omega)$ such that for all $x\neq y \in \Omega$
\begin{equation} \label{relationship Kt and Ks} K_{t}(x,y) \leq c |x-y|^{s-t} K_{s}(x,y). \end{equation}
\end{proposition}

\begin{proof}
See \cite{CaSt} Theorems 2.3 and 2.4.  

Theorem 2.4 in \cite{CaSt} does not explicitly state the result \eqref{relationship Kt and Ks}.  However, it does state that for each kernel $K_s$ there exists a constant $c_s$ dependent on $s$ and $\Omega$ such that 
\[ \frac{1}{c_s} |x-y|^{2+s} K_s(x,y) \leq \min\paren{1,\frac{\eigen{0}(x)\eigen{0}(y)}{|x-y|^2}} \leq c_s |x-y|^{2+s} K_s(x,y). \]
Since the middle term does not depend on $s$, we can say that
\[ |x-y|^{2+t} K_t(x,y) \leq c_t c_s |x-y|^{2+s} K_s(x,y) \]
from which \eqref{relationship Kt and Ks} follows.  
\end{proof}

From the explicit formulae given in \cite{CaSt}, we see that $K_{2s}$ is approximately equal to the standard kernel for the $\R^2$ fractional Laplacian $(-\Laplace)^s$ when both $x$ and $y$ are in the interior of $\Omega$ or when $x$ and $y$ are extremely close together, but decays to zero when one point is in the interior and the other is near the boundary.  The kernel $B_{2s}$ is well-behaved in the interior but has a singularity at the boundary $\del \Omega$.  This justifies our thinking of the $K_{2s}$ term as the interior term and $B_{2s}$ as a boundary term.  

When comparing the computations in this paper to corresponding computations on $\R^2$, one finds that the interior term behaves nearly the same as in the unbounded case, while the boundary term behaves roughly like a lower order term (in the sense that it is easily localized).  

Many useful results can be derived from Caffarelli-Stinga representation formula.  We summarize them in the following lemma.  

\begin{lemma} \label{thm:Lambda stuff}
Let $\Omega \subseteq \R^2$ be a bounded open set with $C^{2,\beta}$ boundary for some $\beta \in (0,1)$.  

\begin{enumerate}[(a)]
\item \label{thm:disjoint} Let $s \in (0,1)$.  If $f$ and $g$ are non-negative functions with disjoint support (i.e. $f(x)g(x) = 0$ for all $x$), then 
\[ \int \Lambda^s f \Lambda^s g \,dx \leq 0. \]

\item \label{thm:product rule} Let $s \in (0,1)$.  If $g \in C^{0,1}(\Omega)$ then for some constant $C = C(s)$ independent of $\Omega$
\[ \norm{fg}_{\HD^s} \leq 2 \norm{g}_\infty \norm{f}_{\HD^s} + C \norm{f}_2 \sup_y \int \frac{|g(x)-g(y)|^2}{|x-y|^{2+2s}} \,dx. \]

\item \label{thm:extra product rule} Let $s \in (0,1)$.  If $g \in C^{0,1}(\Omega)$ then for some constant $C=C(s)$ independent of $\Omega$
\[  \norm{fg}_{\HD^s} \leq C \norm{g}_{C^{0,1}(\Omega)} \paren{\norm{f}_2 + \norm{f}_{\HD^s}}. \]

\item \label{thm:L1 of Lambda bounded} Let $s\in(0,1/2)$.  Let $g$ an $L^\infty(\Omega)$ function and $f \in \HD^{2s}$ be non-negative with compact support.  Let $\Comega$ be a constant such that
\begin{equation} \label{K bounded between orders} K_s(x,y) \leq \Comega |x-y|^{3s} K_{4s}(x,y). \end{equation}

Then there exists a constant $C$ depending only on $s$ and $\Comega$ such that
\[ \int \Lambda^{s/2} g \Lambda^{s/2} f \leq C \norm{g}_\infty |\supp(f)|^{1/2} \paren{ \norm{f}_2 + \norm{f}_{\HD^{2s}}}. \]

\item \label{thm:L1 of Lambda1/4 bounded} Let $g$ an $L^\infty(\Omega)$ function and $f \in \HD^{1/2}$ be non-negative with compact support.  Let $\Comega$ be a constant such that
\[ K_{1/4}(x,y) \leq \Comega |x-y|^{3/4} K_{1}(x,y). \]

Then there exists a constant $C$ depending only on $\Comega$ such that
\[ \int g \Lambda^{1/4} f \leq C \norm{g}_\infty |\supp(f)|^{1/2} \paren{ \norm{f}_2 + \norm{f}_{\HD^{1/2}}}. \]

\end{enumerate}
\end{lemma}

\begin{proof}
We prove these corollaries one at a time.  

\textbf{Proof of \eqref{thm:disjoint}:}
From Proposition \ref{thm:Caff Stinga representation}
\[ \int \Lambda^s f \Lambda^s g \,dx = \iint [f(x)-f(y)][g(x)-g(y)] K(x,y) \,dxdy + \int f(x) g(x) B(x) \,dx. \]

Since $f$ and $g$ are non-negative and disjoint, the $B$ term vanishes.  Moreover, the product inside the $K$ term becomes
\[ [f(x)-f(y)][g(x)-g(y)] = -f(x)g(y)-f(y)g(x) \leq 0. \]
Since $K$ is non-negative, the result follows.  

\vskip.3cm
\textbf{Proof of \eqref{thm:product rule}:}
From Proposition \ref{thm:Caff Stinga representation}
\[ \int |\Lambda^s (fg)|^2 = \iint \paren{g(x)[f(x)-f(y)] + f(y)[g(x)-g(x)]}^2 K + \int f^2 g^2 B \]
\[ \leq 2 \norm{g}_\infty^2 \norm{f}_{\HD^s}^2 + C(s) \int f(y)^2 \int \frac{|g(x)-g(y)|^2}{|x-y|^{2+2s}} \,dxdy. \]

\vskip.3cm
\textbf{Proof of \eqref{thm:extra product rule}:}
This follows immediately from \eqref{thm:product rule}, since
\[ |g(x)-g(y)| \leq \paren{\norm{g}_\infty} \wedge \paren{\norm{\grad g}_\infty |x-y|} \]
and 
\[  \int \frac{1 \wedge |x-y|^2}{|x-y|^{2+2s}} \,dx \]
is bounded uniformly in $y$.  

\vskip.3cm
\textbf{Proof of \eqref{thm:L1 of Lambda bounded}:}
From Proposition \ref{thm:Caff Stinga representation} we can decompose
\[ \int \Lambda^{s/2} g \Lambda^{s/2} f = \Rom{1}_< + \Rom{1}_\geq + \Rom{2} \]
where
\begin{align*} 
\Rom{1}_< &:= \iint_{|x-y| < 1} [g(x)-g(y)][f(x)-f(y)] K_s, \\
\Rom{1}_\geq &:= \iint_{|x-y|\geq 1} [g(x)-g(y)][f(x)-f(y)] K_s, \\
\Rom{2} &:= \int f g B_s. 
\end{align*}

First we estimate $\Rom{1}_<$.  From \eqref{K bounded between orders} and from the symmetry of the integrand and the fact that $[f(x)-f(y)]$ vanishes unless at least one of $f(x)$ or $f(y)$ is non-zero,
\[ \abs{\Rom{1}_<} \leq 2 \iint_{|x-y| < 1} \indic{f > 0}(x) \abs{g(x)-g(y)} \cdot \abs{f(x)-f(y)} \cdot |x-y|^{3s} K_{4s}. \]
We can break this up by H\"{o}lder's inequality
\[ \abs{\Rom{1}_<} \leq 2 \paren{\iint_{|x-y| < 1} \indic{f > 0}(x) [g(x)-g(y)]^2 |x-y|^{6s} K_{4s} }^{1/2} \paren{\iint [f(x)-f(y)]^2 K_{4s} }^{1/2}. \]
The kernel $|x-y|^{6s} K_{4s} \indic{|x-y| < 1}$ is integrable in $y$ for $x$ fixed.  Therefore
\begin{equation} \label{K near of L1 less than L2 stuff} \abs{\Rom{1}_<} \leq 2 \paren{ \paren{2\norm{g}_\infty}^2 \int C \indic{f > 0}(x) \,dx }^{1/2} \paren{ \norm{f}_{\HD^{2s}}^2}^{1/2}. \end{equation}
\vskip.3cm

For the term $\Rom{1}_\geq$, by the symmetry of the integrand we have
\[ \abs{\Rom{1}_\geq} \leq  2\norm{g}_\infty 2\int |f(x)| \int_{|x-y|\geq 1} K_s(x,y)dy \,dx. \]
Since $K_s \indic{|x-y|\geq 1}$ is integrable in $y$ for $x$ fixed,
\begin{equation} \label{K far of L1 less than L2 stuff}
\abs{\Rom{1}_\geq} \leq C \norm{g}_\infty \norm{f}_1.  
\end{equation}
\vskip.3cm

For the boundary term \Rom{2}, 
\[ \abs{\Rom{2}} \leq \norm{g}_\infty \int\indic{f>0} f B_s. \]
Since $f \geq 0$, $[f(x)-f(y)][\indic{f > 0}(x) - \indic{f > 0}(y)] \geq 0$.  Therefore
\[ \int \indic{f > 0} f B_s \leq \int \Lambda^{s/2} \indic{f>0} \Lambda^{s/2} f = \int \indic{f>0} \Lambda^{s} f. \]
Applying H\"{o}lder's inequality, we arrive at
\[ \abs{\Rom{2}} \leq \norm{g}_\infty |\supp(f)|^{1/2} \norm{f}_{\HD^{s}}. \]
This combined with \eqref{K near of L1 less than L2 stuff} and \eqref{K far of L1 less than L2 stuff} gives us 
\[ \int \Lambda^{s/2} g \Lambda^{s/2} f \leq C \norm{g}_\infty \paren{|\supp(f)|^{1/2} \norm{f}_{\HD^{2s}} + \norm{f}_1 + |\supp(f)|^{1/2} \norm{f}_{\HD^s}}. \]

The lemma follows since $\norm{f}_1 \leq |\supp(f)|^{1/2} \norm{f}_2$ and since $\norm{f}_{\HD^s} \leq \norm{f}_{L^2} + \norm{f}_{\HD^{2s}}$.  

\vskip.3cm
\textbf{Proof of \eqref{thm:L1 of Lambda1/4 bounded}:}
This is an immediate application of part \eqref{thm:L1 of Lambda bounded}.  

\end{proof}

Let us consider the relationship between the norm $\HD^s$ and the $H^s$ norm on $\R^2$.  

It is known (see \cite{CoIg.sqg} and \cite{CaSt}) that for $s \in (0,1)$ the spaces $\HD^s$ are equivalent to certain subsets of $H^s(\Omega)$ spaces defined in terms of the Gagliardo semi-norm.  In particular, we know that smooth functions with compact support are dense in $\HD^s$ for $s \in [0,1]$ and that elements of $\HD^s$ have trace zero for $s \in [\frac{1}{2},1]$.  
\vskip.3cm

The most important fact for us is that the fractional Sobolev norms defined in terms of extension are dominated by our $\HD^s$ norm with a constant that is independent of $\Omega$.  

We do not claim that this result is new, but we present a detailed proof because the result is crucial to the De Giorgi method.  The De Giorgi lemmas require Sobolev embeddings and Rellich-Kondrachov embeddings which are independent of scale.  
\vskip.3cm

Define the extension-by-zero operator $E:L^2(\Omega) \to L^2(\R^2)$
\[ E f (x) = \begin{cases} f(x) & x \in \Omega, \\ 0 & x \in \R^2 \setminus \Omega. \end{cases} \]

\begin{proposition} \label{thm:hadamard 3 lines}
Let $\Omega \subseteq \R^2$ be any bounded open set with $C^{2,\beta}$ boundary for some $\beta \in (0,1)$.  For any $s \in [0,1]$ and function $f \in \HD^s$,
\[ \int_{\R^2} \abs{\paren{-\Laplace}^{s/2} E f}^2 \leq \int_\Omega \abs{\Lambda^s f}^2. \]
Here $\paren{-\Laplace}^s$ is defined in the fourier sense.  
\end{proposition}

We will prove this proposition by interpolating between $s = 0$ and $s=1$.  Before we can do this, we must prove the same in the $s=1$ case.  This result is known (see e.g. Jerison and Kenig \cite{JeKe}) but we include the proof for completeness.  

\begin{lemma} \label{thm:H1 and H1}
Let $\Omega \subseteq \R^2$ be any bounded open set with Lipschitz boundary.  For all functions $f$ in $\HD^1$,
\[ \int_\Omega \abs{\grad f}^2 = \int_\Omega \abs{\Lambda f}^2. \]
\end{lemma}

\begin{proof}
Let $\eigen{i}$ and $\eigen{j}$ be two eigenfunctions of the Dirichlet Laplacian on $\Omega$.  Note that these functions are smooth in the interior of $\Omega$ and vanish at the boundary, so we can apply the divergence theorem and find
\[ \int \grad \eigen{i} \cdot \grad \eigen{j} = - \int \eigen{i} \Laplace \eigen{j} = \lambda_j \int \eigen{i} \eigen{j} = \lambda_j \delta_{i=j}. \]

Consider a function $f = \sum f_k \eigen{k}$ which is an element of $\HD^1$, by which we mean $\sum \lambda_k f_k^2 < \infty$.  Since $\norm{\grad \eigen{k}}_{L^2(\Omega)} = \sqrt{\lambda_k}$, the following sums all converge in $L^2(\Omega)$ and hence the calculation is justified:
\begin{align*}
\int \abs{\grad f}^2 &= \int \paren{\sum_i f_i \grad \eigen{i} } \paren{\sum_j f_j \grad \eigen{j}}
\\ &= \int \sum_{i,j} (f_i f_j) \grad \eigen{i} \cdot \grad \eigen{j}
\\ &= \sum_{i,j} (f_i f_j) \int \grad \eigen{i} \cdot \grad \eigen{j}
\\ &= \sum_j \lambda_j f_j^2.
\end{align*}
From this the result follows.
\end{proof}

We come now to the proof of Proposition \ref{thm:hadamard 3 lines}.  The proof is by complex interpolations using the Hadamard three-lines theorem.  
\begin{proof}
Let $g$ be any Schwartz function in $L^2(\R^2)$, and let $f$ be a function in $\HD^s$.  
Define the function
\[ \Phi(z) = \int_{\R^2} \paren{-\Laplace}^{z/2} g E \Lambda^{s-z} f, \qquad z \in \C, \Re(z) \in [0,1]. \]

Recall (see e.g. \cite{JeKe}) that when $t \in \R$, $(-\Laplace)^{i t}$ is a unitary transformation on $L^2(\R^2)$, and $\Lambda^{i t}$ is a unitary transformation on $L^2(\Omega)$.  

When $\Re(z) = 0$, then $\norm{\paren{-\Laplace}^{z/2} g}_2 = \norm{g}_2$ and $\norm{\Lambda^{s-z} f}_2 = \norm{f}_{\HD^s}$.  Hence
\[ \Phi(z) \leq \norm{g}_2 \norm{f}_{\HD^s}, \qquad \Re(z)=0. \]

When $\Re(z)=1$, integrate by parts to obtain
\[ \Phi(z) = \int_{\R^2} \paren{-\Laplace}^{(z-1)/2} g  \paren{-\Laplace}^{1/2} E \Lambda^{s-z} f. \]
Then $\norm{\paren{-\Laplace}^{(z-1)/2} g}_2 = \norm{g}_2$, while $\norm{\Lambda^{s-z} f}_{\HD^1} = \norm{f}_{\HD^s}$.  As an $\HD^1$ function, $\Lambda^{s-z} f$ has trace zero so 
\[ \norm{\grad E \Lambda^{s-z} f}_{L^2(\R^2)} = \norm{\grad \Lambda^{s-z} f}_{L^2(\Omega)} = \norm{f}_{\HD^s}. \]
Of course $\norm{\paren{-\Laplace}^{1/2} \,\cdot\,}_{L^2(\R^2)} = \norm{\grad \,\cdot\,}_{L^2(\R^2)}$ in general so
\[ \Phi(z) \leq \norm{g}_2 \norm{f}_{\HD^s}, \qquad \Re(z)=1. \]
\vskip.3cm

In order to apply the Hadamard three-lines theorem, we must show that $\Phi$ is differentiable in the interior of its domain.  

Rewrite the integrand of $\Phi$ as
\[ \Four\n\paren{ |\xi|^z \hat{g} } E \sum_k \lambda_k^{\frac{s-z}{2}} f_k. \]
The derivative $\ddz$ commutes with linear operators like $\Four\n$ and $E$, so the derivative is
\begin{equation} \label{derivative of integrand} \Four\n\paren{ \ln(|\xi|) |\xi|^z \hat{g} } E \sum_k \lambda_k^{\frac{s-z}{2}} f_k + \Four\n\paren{ |\xi|^z \hat{g} } E \sum_k \frac{-1}{2} \ln(\lambda_k) \lambda_k^{\frac{s-z}{2}} f_k. \end{equation}

Fix some $z \in \C$ with $\Re(z) \in (0,1)$.  Since $g$ is a Schwartz function, $\ln(|\xi|)|\xi|^z \hat{g}$ is in $L^2$.  Moreover, for any $\eps>0$ we have $\ln(\lambda_k) \lambda_k^{\frac{s-z}{2}} \leq C \lambda_k^{\frac{s-z+\eps}{2}}$ for some $C$ independent of $k$ but dependent on $z$, $\eps$.  Take $\eps < \Re(z)$ and, since $f \in \HD^{s}$, this sum will converge in $L^2$.  

The differentiated integrand \eqref{derivative of integrand} is therefore a sum of two products of $L^2$ functions.  In particular it is integrable, which means we can interchange the integral sign and the derivative $\ddz$ and prove that $\Phi'(z)$ is finite for all $0<\Re(z)<1$. 
\vskip.3cm

By the Hadamard three-lines theorem, for any $z \in (0,1)$ we have $\Phi(z) \leq \norm{g}_2 \norm{f}_{\HD^s}$.  
Evaluating $\Phi(s)$, we see
\[ \int_{\R^2} \paren{-\Laplace}^{s/2} g E f \leq \norm{g}_{L^2(\R^2)} \norm{f}_{\HD^s}. \]
This inequality holds for any Schwartz function $g \in L^2(\R^n)$ and any $f \in \HD^{s}$.  

Since Schwartz functions are dense in $L^2(\R^2)$ and $(-\Laplace)^{s/2}$ is self-adoint, the proof is complete.  
\end{proof}

\vskip1cm
\section{Existence of suitable solutions} \label{sec:suitable}

In this section, we define the needed notion of suitable solutions.  This involves two families of localized energy inequalities.  The first family \eqref{leray-hopf condition} concerns the time evolution of $\int_\Omega (\theta - \Psi)_+^2$ for generic cutoff functions $\Psi$.  We need also to control the time derivative $\del_t (\theta-\Psi)_+^2$ in the sense of distributions for the second De Giorgi lemma (see Proposition~\ref{thm:DG2}, step 2).  This control comes in the family of inequalities \eqref{local condition}.  

It is important that the universal constant $\Csuit$ appearing in the suitability conditions \eqref{leray-hopf condition} and \eqref{local condition} is independent of $\Omega$.  The De Giorgi argument requires that we apply the same bound iteratively as we rescale the solution, so our bounds must be scale independent.  For this reason, we will define the constant through Proposition~\ref{thm:suitability} before stating the definition of suitable solutions.  As with the Navier-Stokes equations, it is not obvious that weak solutions constructed directly from the Galerkin scheme are suitable.  Therefore we will construct our weak solutions as vanishing viscosity limits of
\begin{equation} \label{eq:viscous} \begin{cases}
\del_t \theta + u\cdot \grad \theta + \Lambda \theta = \eps \Laplace \theta & (0,\infty) \times \Omega,
\\ u = \grad^\perp \Lambda\n \theta & [0,\infty) \times \Omega,
\\ \theta = \theta_0 & \{0\} \times \Omega.
\end{cases} \end{equation}
The construction of solutions to \eqref{eq:viscous} will follow the Galerkin method (as in \cite{CoIg.fraclap}).  
\vskip.3cm

We begin by defining the universal constant $\Csuit$ and simultaneously showing that the inequalities \eqref{leray-hopf condition} and \eqref{local condition} are valid for sufficiently smooth solutions to the linear equation
\begin{equation} \label{eq:viscous linear}
\begin{cases}
\del_t \theta + u \cdot \grad \theta + \Lambda \theta = \eps \Laplace \theta, \\
\div u = 0,
\end{cases}
\end{equation}
uniformly with respect to $\eps \in [0,1]$.  This smoothness requirement will be shown to be valid when $\eps > 0$.  

\begin{proposition}[Energy Inequalities] \label{thm:suitability}
There exists a universal constant $\Csuit$ such that the following holds:

Let $\Omega \subseteq \R^2$ be bounded and open with $C^{2,\beta}$ boundary, $\beta \in (0,1)$, and let $0 < T < \infty$ a time, and let $\eps \in [0,1]$.  Let $\theta, u$ be a solution to \eqref{eq:viscous linear} on $\Omega \times [0,T]$, with $\theta \in L^\infty(0,T; L^2(\Omega)) \cap L^2(0,T; H_0^1(\Omega))$ and $u \in L^\infty(0,T; L^2(\Omega)) \cap L^4(0,T; L^4(\Omega))$.  

Then for any smooth non-negative function $\Psi \in C^\infty([0,\infty)\times \R^2)$ satisfying $\norm{\grad\Psi}_{L^\infty([0,\infty)\times\R^2)} \leq k$ and the H\"{o}lder seminorm $\sup_{[0,\infty)} \bracket{\Psi(t,\cdot)}_{1/4; \R^2} \leq k$ for some constant $k$, any time $S \in (0,T)$, and any smooth non-negative $\varphi \in \Ctest(S,T;C^\infty(\Omega))$, the function $\theta_+ := \paren{\theta - \Psi}_+$ satisfies
\begin{equation}\label{leray-hopf condition}
\ddt \int \theta_+^2 + \int \abs{\Lambda^{1/2} \theta_+}^2 \leq \Csuit \paren{ k^2 \int \indic{\theta \geq \Psi} + \abs{\int \theta_+ (\del_t \Psi + u\cdot\grad\Psi)} } \qquad \forall t \in [0,T]
\end{equation}
in the sense of distributions and
\begin{equation} \label{local condition} 
\begin{aligned} 
\frac{-1}{2} \int_S^T\!\!\!\!\int \theta_+^2 \del_t \varphi &\leq \frac{1}{2} \int_S^T\!\!\!\!\int  \theta_+^2 u \cdot \grad \varphi - \int_S^T\!\!\!\!\int \varphi \theta_+ \paren{\del_t \Psi + u \cdot \grad \Psi} 
+ \Csuit \norm{\varphi}_{C^0(S,T;C^2)} \bigg( \paren{1\!\!+\!\!\frac{1}{S}} \int_0^T\!\!\!\!\int \theta_+^2 
\\ & \hskip2cm + k^2 \int_0^T\!\!\!\!\int \indic{\theta\geq\Psi} + \int_0^T\! \abs{\int \theta_+ \paren{\del_t \Psi + u\cdot\grad\Psi}} \bigg).
\end{aligned}
\end{equation}
\end{proposition}

\begin{remark}
Note that since $\Csuit$ is universal, Proposition~\ref{thm:suitability} does not depend on the values of $\norm{\theta}_{L^\infty(L^2)}$, $\norm{\theta}_{L^2(H_0^1)}$, $\norm{u}_{L^\infty(L^2)}$, or $\norm{u}_{L^4(L^4)}$, but only on the fact that these quantities are finite.  Therefore, using the natural scaling of \eqref{eq:viscous linear}, if $(\theta, u)$ verify the assumptions of Proposition~\ref{thm:suitability} on $[0,T]\times\Omega$, then so does $\big(\lambda \theta(\mu \cdot, \mu \cdot), u(\mu\cdot, \mu\cdot)\big)$ on $[0,\mu\n T] \times \mu\n \Omega$, for any $\lambda \in \R$ and $\mu > 0$ such that $\mu\n \eps \in [0,1]$.  Therefore the proposition applies also to these scaled functions, with the same universal constant $\Csuit$.  
\end{remark}

\begin{proof}[Proof of Proposition \ref{thm:suitability}]
Since $\theta_+ \in L^\infty(0,T; L^2(\Omega)) \cap L^2(0,T; H_0^1(\Omega))$, we can multiply \eqref{eq:viscous linear} by $\theta_+$ and integrate in space to obtain
\[ 0 = \int \theta_+ \bracket{ \del_t + u \cdot \grad + \Lambda - \eps \Laplace} \paren{\theta_+ + \Psi - \theta_-} \]
which decomposes into three terms, corresponding to $\theta_+$, $\Psi$, and $\theta_-$.  We analyze them one at a time.  

Firstly,
\begin{align*} 
\int \theta_+ \bracket{ \del_t + u \cdot \grad + \Lambda -\eps\Laplace} \theta_+ &= \paren{\frac{1}{2}} \ddt \int \theta_+^2 + \paren{\frac{1}{2}} \int \div u \, \theta_+^2 + \int \abs{\Lambda^{1/2} \theta_+}^2 + \eps \int \abs{\grad\theta_+}^2
\\ &= \paren{\frac{1}{2}} \ddt \int \theta_+^2 + \int \abs{\Lambda^{1/2} \theta_+}^2 + \eps \int \abs{\grad\theta_+}^2.
\end{align*}

The $\Psi$ term produces important error terms:
\begin{align*} 
\int \theta_+ \bracket{ \del_t + u \cdot \grad + \Lambda -\eps\Laplace} \Psi &= \int \theta_+ \del_t \Psi + \int \theta_+ u \cdot \grad \Psi + \int \Lambda^{1/2} \theta_+ \Lambda^{1/2} \Psi + \eps \grad\theta_+ \cdot \grad\Psi
\\ &= \int \theta_+ (\del_t \Psi + u \cdot \grad \Psi) + \int \Lambda^{1/2} \theta_+ \Lambda^{1/2} \Psi + \eps \grad\theta_+ \cdot \grad\Psi.
\end{align*}

Since $\theta_+$ and $\theta_-$ have disjoint support, the $\theta_-$ term is nonnegative by Lemma~\ref{thm:Lambda stuff} part \eqref{thm:disjoint}:
\begin{align*} 
\int \theta_+ \bracket{ \del_t + u \cdot \grad + \Lambda } \theta_- &= \paren{\frac{1}{2}} \int \theta_+ \del_t \theta_- + \int \theta_+ u \cdot \grad \theta_- + \int \Lambda^{1/2} \theta_+ \Lambda^{1/2} \theta_- + \eps\int \grad\theta_+ \grad\theta_-
\leq 0.
\end{align*}

Put together, we arrive at 
\begin{equation} \label{first form of cacciopolli} \paren{\frac{1}{2}} \ddt \int \theta_+^2 + \int \abs{\Lambda^{1/2} \theta_+}^2 + \int \Lambda^{1/2} \theta_+ \Lambda^{1/2} \Psi \leq -\int \theta_+ (\del_t \Psi + u \cdot \grad \Psi) - \eps \bracket{\int \grad\theta_+ \cdot \grad\Psi + \int \abs{\grad\theta_+}^2}. \end{equation}

The $\eps$ term is bounded, using the fact that $\grad\theta_+ = \indic{\theta_+ > 0} \grad\theta_+$ and $\eps \in [0,1]$, by
\begin{equation} \label{eps term for leray hopf} \begin{aligned}
-\eps\bracket{\int \abs{\grad\theta_+}^2 + \int \grad\theta_+ \cdot \grad\Psi} &\leq \frac{-\eps}{2} \int |\grad\theta_+|^2 + \frac{\eps}{2} \int |\grad\Psi|^2 \indic{\theta_+ > 0}^2
\\ &\leq \frac{k^2}{2} \int \indic{\theta_+ > 0}.
\end{aligned} \end{equation}

At this point we break down the $\Lambda^{1/2} \theta_+ \Lambda^{1/2} \Psi$ term using the formula from Proposition~\ref{thm:Caff Stinga representation}.  
\[ \int \Lambda^{1/2} \theta_+ \Lambda^{1/2} \Psi = \iint [\theta_+(x)-\theta_+(y)][\Psi(x)-\Psi(y)] K(x,y) + \int \theta_+ \Psi B. \]
Since $B \geq 0$ and $\Psi$ is non-negative by assumption, the $B$ term is non-negative and so
\begin{equation} \label{just the K cacciopolli} \int \Lambda^{1/2} \theta_+ \Lambda^{1/2} \Psi \geq \iint [\theta_+(x)-\theta_+(y)][\Psi(x)-\Psi(y)] K(x,y). \end{equation}
The remaining integral is symmetric in $x$ and $y$, and the integrand is only nonzero if at least one of $\theta_+(x)$ and $\theta_+(y)$ is nonzero.  Hence
\[ \abs{\iint [\theta_+(x)-\theta_+(y)][\Psi(x)-\Psi(y)] K(x,y)} \leq 2 \iint \indic{\theta_+>0}(x) \abs{\theta_+(x)-\theta_+(y)} \cdot \abs{\Psi(x)-\Psi(y)} K(x,y). \]
Now we can break up this integral using Young's inequality, and since $\iint [\theta_+(x)-\theta_+(y)]^2K \leq \norm{\theta_+}_{\HD^{1/2}}^2$ the inequality \eqref{just the K cacciopolli} becomes
\begin{equation} \label{good lower bound on psi term cacciopoli}
\int \Lambda^{1/2} \theta_+ \Lambda^{1/2} \Psi + \frac{1}{2} \int \abs{\Lambda^{1/2}\theta_+}^2 \geq - 2 \iint \indic{\theta_+>0}(x) [\Psi(x)-\Psi(y)]^2 K(x,y). 
\end{equation}

It remains to bound the quantity $[\Psi(x)-\Psi(y)]^2 K(x,y)$.  By Proposition~\ref{thm:Caff Stinga representation}, there is a universal constant $C$ such that
\[ K(x,y) \leq \frac{C}{|x-y|^{3}}. \]
The cutoff $\Psi$ is locally Lipschitz, and H\"{o}lder continuous with exponent $1/4$, by assumption.  Therefore 
\[ [\Psi(x)-\Psi(y)]^2 K(x,y) \leq Ck^2 |x-y|^{-1} \wedge |x-y|^{-2.5}. \]
Since $1 < 2 < 2.5$, this quantity is integrable.  Thus
\[ \int \indic{\theta_+>0}(x) \int [\Psi(x)-\Psi(y)]^2 K(x,y) \,dydx \leq C k^2 \int \indic{\theta_+>0} \,dx. \]
Combining this with \eqref{first form of cacciopolli}, \eqref{good lower bound on psi term cacciopoli}, and \eqref{eps term for leray hopf} we obtain \eqref{leray-hopf condition}.  

\vskip.3cm

We begin now the proof of \eqref{local condition}.  Since $\theta_+ \in L^\infty(0,T; L^2(\Omega)) \cap L^2(0,T; H_0^1(\Omega))$, by interpolation we can further conclude $\theta_+,u \in L^4(0,T; L^4(\Omega))$.  Therefore we can multiply \eqref{eq:viscous linear} by $\varphi \theta_+$ and integrate in space to obtain
\[ 0 = \int \varphi \theta_+ \bracket{ \del_t + u \cdot \grad + \Lambda } \paren{\theta_+ + \Psi - \theta_-} \]
which decomposes into three terms, corresponding to $\theta_+$, $\Psi$, and $\theta_-$.  After rearranging and integrating by parts, this becomes
\begin{equation}\label{equality with varphi}
\frac{1}{2} \int \varphi \del_t \theta_+^2 = \frac{1}{2} \int  \theta_+^2 u \cdot \grad \varphi - \int \varphi \theta_+ \paren{\del_t \Psi + u \cdot \grad \Psi} - \int \varphi \theta_+ \Lambda \theta_+ - \int \varphi \theta_+ \Lambda \Psi + \eps \int \varphi \theta_+ \Laplace (\theta_+ + \Psi). 
\end{equation}

The $\eps$ term decomposes as 
\begin{equation} \begin{aligned} \label{eps term for local}
\eps \int \varphi \theta_+ \Laplace (\theta_+ + \Psi) &= -\eps \int \varphi \grad \theta_+ \cdot \grad (\theta_+ + \Psi) - \eps \int \theta_+ \grad \varphi \cdot \grad (\theta_+ + \Psi)
\\ &= -\eps \int \varphi \abs{\grad \theta_+}^2 - \eps \int \varphi \grad \theta_+ \cdot \grad \Psi + \frac{\eps}{2} \int \theta_+^2 \Laplace \varphi - \eps \int \theta_+ \grad\varphi \cdot \grad \Psi
\\ &\leq  \frac{\eps}{2} \int \varphi \abs{\grad\Psi}^2 \indic{\theta_+ > 0} + \frac{\eps}{2} \int \theta_+^2 \Laplace \varphi + \frac{\eps}{2} \int \theta_+^2 |\grad\varphi| + \frac{\eps}{2} \int \indic{\theta_+ > 0} |\grad \varphi| |\grad \Psi|^2
\\ &\leq k^2 \norm{\varphi}_{C^1} \int \indic{\theta_+ > 0} + \norm{\varphi}_{C^2} \int \theta_+^2.
\end{aligned} \end{equation}

The $\int \varphi \theta_+ \Lambda \theta_+$ term is bounded by Lemma~\ref{thm:Lambda stuff} part \eqref{thm:extra product rule}
\begin{equation} \label{phi disrupting theta_+ theta_+}
- \int \varphi \theta_+ \Lambda \theta_+ \leq C \norm{\varphi}_{C^1} \paren{ \int \theta_+^2 + \int \abs{\Lambda^{1/2} \theta_+}^2 }
\end{equation}
and the $\int \varphi \theta_+ \Lambda \Psi$ term is bounded, just as for the $\int \theta_+ \Lambda \Psi$ term in the previous family of inequalities but with the addition of Lemma~\ref{thm:Lambda stuff} part \eqref{thm:extra product rule},
\begin{equation} \label{phi disrupting Lambda Psi} \begin{aligned}
- \int \varphi \theta_+ \Lambda \Psi &\leq \iint [\varphi(x) \theta_+(x) - \varphi(y) \theta_+(y)][\Psi(x)-\Psi(y)] K_1
\\ &\leq 2\iint \indic{\theta_+>0} |\varphi(x) \theta_+(x) - \varphi(y) \theta_+(y)|\cdot|\Psi(x)-\Psi(y)| K_1
\\ &= 2 \iint \paren{ \norm{\varphi}_{C^1}^{-1/2} |\varphi(x) \theta_+(x) - \varphi(y) \theta_+(y)|} \paren{\norm{\varphi}_{C^1}^{1/2} \indic{\theta_+ > 0} |\Psi(x)-\Psi(y)|} K_1
\\ &\leq \norm{\varphi}_{C^1}\n \norm{\varphi \theta_+}_{\HD^{1/2}}^2 + \norm{\varphi}_{C^1} \iint \indic{\theta_+>0} [\Psi(x)-\Psi(y)]^2 K_1
\\ &\leq C \norm{\varphi}_{C^1} \paren{\int \theta_+^2 + \int \abs{\Lambda^{1/2}\theta_+}^2 } + C k^2 \norm{\varphi}_{C^1} \paren{\int_{\R^2} |y|^{-1} \wedge |y|^{-2.5} \,dy} \int \indic{\theta_+>0}.
\end{aligned} \end{equation}

From the inequality \eqref{leray-hopf condition} already proven, we can obtain by a standard argument
\begin{equation} \label{dissipation controlled by earlier in time} 
\int_S^T\!\! \int \abs{\Lambda^{1/2} \theta_+}^2 \leq \frac{1}{S} \int_0^T\!\! \int \theta_+^2 + k^2 \int_0^T\!\! \int \indic{\theta_+>0} + \int_0^T\! \abs{\int \theta_+ \paren{\del_t \Psi + u\cdot\grad\Psi}}. 
\end{equation}

By combining \eqref{equality with varphi} with \eqref{eps term for local}, \eqref{phi disrupting theta_+ theta_+}, and \eqref{phi disrupting Lambda Psi} we obtain
\[ \frac{1}{2} \int \varphi \del_t \theta_+^2 = \frac{1}{2} \int  \theta_+^2 u \cdot \grad \varphi - \int \varphi \theta_+ \paren{\del_t \Psi + u \cdot \grad \Psi} + C \norm{\varphi}_{C^2} \paren{ \int \theta_+^2 + k^2 \int \indic{\theta_+>0} + \int \abs{\Lambda^{1/2}\theta_+}^2 }. \]
Integrating this inequality from $S$ to $T$ and applying \eqref{dissipation controlled by earlier in time}, we obtain \eqref{local condition}.  
\end{proof}
\vskip.3cm

We will construct global-in-time solutions to \eqref{eq:main linear} (equivalently \eqref{eq:viscous linear} with $\eps = 0$) for any initial value $\theta_0 \in L^2$ which verify these energy inequalities \eqref{leray-hopf condition} and \eqref{local condition} with the same universal constant $\Csuit$ at all scales, but which may not be in $L^2(H_0^1)$.  
\begin{definition}
A pair $\theta, u$ is called a \textbf{suitable solution} to \eqref{eq:main linear} on $[0,\infty) \times \Omega$ if $\Omega \subseteq \R^2$ open and bounded, $\theta,u \in L^\infty(\R_+;L^2(\Omega))$, $\theta \in L^2(\R_+;\HD^{1/2}(\Omega))$, $u \in L^3(\R_+; L^3(\Omega))$ and $\theta, u$ is a suitable solution on $[0,T]\times\Omega$ for all $0 < T < \infty$.  

A pair $\theta, u$ is called a \textbf{suitable solution} to \eqref{eq:main linear} on a space time domain $[0,T]\times \Omega$ if $T <\infty$, $\Omega \subseteq\R^2$ open and bounded,  $\theta,u \in L^\infty(0,T;L^2(\Omega))$, $\theta \in L^2(0,T;\HD^{1/2}(\Omega))$, $u \in L^3(0,T; L^3(\Omega))$ and
\begin{enumerate}
\item $\theta$, $u$ solve \eqref{eq:main linear} in the sense of distributions on $[0,T]\times\Omega$, \\
\item $\theta$, $u$ satisfy \eqref{leray-hopf condition} and \eqref{local condition} at all scales with the same universal constant $\Csuit$ defined in Proposition~\ref{thm:suitability}.  More specifically, the following holds:

Let $\lambda \in \R$ and $\mu \in (0,1)$ be given and let $\Psi \in C^\infty([0,\infty)\times \R^2)$ be any smooth non-negative function satisfying $\norm{\grad\Psi}_{L^\infty([0,\infty)\times\R^2)} \leq k$ and $\sup_{[0,\infty)} \bracket{\Psi(t,\cdot)}_{1/4; \R^2} \leq k$ for some constant $k$.  

Define $\tilde{\Omega} := \{x \in \R^2: \mu x \in \Omega\}$, $\tilde{T}:= \mu\n T$, $\tilde{\theta}_+(t,x) := (\lambda \theta(\mu t, \mu x)-\Psi(t, x))_+$, and $\tilde{u}(t,x):= u(\mu t, \mu x)$.  Let $S \in (0,\tilde{T})$ and let $\varphi \in \Ctest(S,T;C^\infty(\tilde{\Omega}))$ be non-negative.  

Then $\tilde{\theta}_+$ and $\tilde{u}$ and $\varphi$ and $\Psi$ satisfy \eqref{leray-hopf condition} and \eqref{local condition} on $\tilde{\Omega}$ with times 0, $S$ and $\tilde{T}$, with the same universal constant $\Csuit$.  
\end{enumerate}
\end{definition}


The rest of this section is dedicated to the proof of the following proposition:

\begin{proposition}[Existence of global suitable solutions] \label{thm:existence}
There exists a universal constant $C > 0$ such that the following holds:

Given an open, bounded domain $\Omega \subseteq \R^2$ with $C^{2,\beta}$ boundary, $\beta \in (0,1)$, and initial data $\theta_0 \in L^2(\Omega)$, there exists a global-in-time weak solution $\theta$ to \eqref{eq:main nonlinear} such that, for any $0 < T < \infty$, $\theta$ and $u := \grad^\perp \Lambda^{-1} \theta$ are a suitable solution to \eqref{eq:main linear} on $[0,T] \times \Omega$.  

Moreover, $\theta \in L^\infty([0,\infty);L^2(\Omega)) \cap L^2([0,\infty);\HD^{1/2}(\Omega))$, and $\theta(t,\cdot) \to \theta_0(\cdot)$ weakly in $L^2(\Omega)$ as $t \to 0$, and for any $S > 0$
\[ \norm{\theta}_{L^\infty([S,\infty)\times\Omega)} \leq \frac{C}{S} \norm{\theta_0}_{L^2(\Omega)}. \]
\end{proposition}

To construct global suitable solutions, we will use the vanishing viscosity method.  First, we must prove existence of global weak solutions to \eqref{eq:viscous}.  

\begin{lemma}[Existence for viscous equation] \label{thm:viscous}
There exists a universal constant $C$ such that the following holds:

Given an open, bounded domain $\Omega \subseteq \R^2$, initial data $\theta_0 \in L^2(\Omega)$ and a constant $\eps > 0$, there exists a global-in-time weak solution $\theta$ to \eqref{eq:viscous}.  

In particular, $\theta \in C^0([0,\infty);L^2(\Omega)) \cap L^2([0,\infty);H_0^1(\Omega))$, $u \in C^0([0,\infty);L^2(\Omega)) \cap L^4([0,\infty)\times\Omega)$, and $\del_t \in L^2([0,\infty);H^{-1}(\Omega))$, and $\theta(t,\cdot) \to \theta_0(\cdot)$ weakly in $L^2(\Omega)$ as $t \to 0$, and for any $S > 0$
\[ \norm{\theta}_{L^\infty\paren{[S,\infty)\times\Omega}} \leq \frac{C}{S} \norm{\theta_0}_{L^2(\Omega)}. \]
\end{lemma}
The proof of existence is by Galerkin's method, while the $L^\infty$ bound uses a De Giorgi argument.  

\begin{proof}
Recall that $\eigen{j}$ are the eigenfunctions of $-\Laplace_D$.  Let $N$ be an integer parameter, and $W_N := \lspan(\eigen{0},\ldots,\eigen{N})$, which consists only of smooth functions which vanish on $\del\Omega$.  We seek first a solution $\theta_N \in W_N$ to the weak equation
\begin{equation} \label{weak eqn for theta_N}
\int \varphi \del_t \theta_N + \int \varphi \grad^{\perp} \Lambda^{-1} \theta_N \cdot \grad \theta_N + \int \varphi \Lambda \theta_N + \eps \int \grad \theta_N \grad \varphi = 0, \qquad \forall t \in \R_{\geq 0}, \varphi \in W_N. 
\end{equation}

If we write
\[ \theta_N(t,x) := \sum_{i=0}^N \alpha_{i,N}(t) \eigen{i}(x) \]
and choose $\varphi = \eigen{i}$ as a test function, then $\theta_N$ solves \eqref{weak eqn for theta_N}  if and only if, for all $i \leq N$,
\[ \alpha'_{i,N}(t) + \sum_{j=0}^N \sum_{k=0}^N \alpha_{j,N}(t) \alpha_{k,N}(t) B_{ijk} + \lambda_i^{1/2} \alpha_{i,N}(t) + \eps \lambda_i \alpha_{i,N}(t) = 0 \]
with
\[ B_{ijk} = \lambda_j^{-1/2} \int \eigen{i} \grad^{\perp} \eigen{j} \cdot \grad \eigen{k} \]
a constant tensor.  

By Peano's existence theorem for ODEs, solutions to this system exist on some interval $[0,T]$ where $T$ depends on $\Omega$ and $N$ and (since $W_N$ is finite dimensional and all norms are equivalent) the $L^2$ norm of the initial data.  

Since $\theta_N \in W_N$ we can take $\theta_N$ as a test function and obtain, for any solution $\theta_N$ to \eqref{weak eqn for theta_N},
\[ \ddt \int \theta_N^2 + \int \abs{\Lambda^{1/2} \theta_N}^2 + \eps \int \abs{\grad \theta_N}^2 = 0. \]
Therefore in particular $\norm{\theta_N}_{L^2(\Omega)}$ is non-increasing in time and we conclude that $\theta_N$ exists for all time.  Moreover, $\theta_N$ is uniformly bounded in $L^\infty(L^2(\Omega))$ and $L^2(H_0^1(\Omega))$.  
\vskip.3cm

To take a limit in $N$, we need uniform regularity in time.  From \eqref{weak eqn for theta_N} we can bound
\[ \int_0^\infty \int \del_t \theta_N \varphi \leq \norm{\theta_N}_{L^4(L^4)}^2 \norm{\varphi}_{L^2(H_0^1)} + \norm{\theta_N}_{L^2(L^2)} \norm{\varphi}_{L^2(H_0^1)} + \norm{\theta_N}_{L^2(H_0^1)} \norm{\varphi}_{L^2(H_0^1)}. \]
Note that $\theta_N$ is uniformly bounded in $L^4(L^4)$ by interpolation and $L^2(L^2)$ by Poincar\'{e}'s inequality.  Therefore $\iint \varphi \del_t \theta_N \leq C \norm{\varphi}_{L^2(H_0^1)}$ for all $\varphi \in L^2(W_N)$ for a constant $C$ independent of $N$.  Since $\del_t \theta_N \in W_N$, this is sufficient to show that $\norm{\del_t \theta_N}_{L^2(H^{-1})}$ is uniformly bounded.  
\vskip.3cm

By Aubin-Lions, we conclude that $\theta_N$ is a compact sequence in $L^2([0,\infty)\times\Omega)$ and so it has an $L^2$ limit $\theta$.  This limit $\theta$ is in $L^\infty(L^2(\Omega))$ and $L^2(H_0^1(\Omega))$ and $\del_t \theta \in L^2(H^{-1}(\Omega))$.  

We must prove that $\theta$ is a weak solution to \eqref{eq:viscous}.  Let $\varphi \in \Ctest(W_M)$ for some $M$.  For $N \geq M$, 
\[ - \iint \theta_N \del_t\varphi - \iint \theta_N \grad^{\perp} \Lambda^{-1} \theta_N \cdot \grad \varphi + \iint \theta_N \Lambda \varphi - \eps \iint \theta_N \Laplace \varphi = 0. \]
This expression is continuous for $\theta_N \in L^2(L^2)$, so by taking $N \to \infty$ we obtain 
\[ - \int \theta \del_t\varphi - \int \theta \grad^{\perp} \Lambda^{-1} \theta \cdot \grad \varphi + \int \theta \Lambda \varphi + \eps \int \grad \theta \cdot \grad \varphi = 0 \]
for any $\varphi \in \Ctest(W_M)$ for any $M \in \N$.  By density, $\theta$ solves \eqref{eq:viscous} in the sense of distributions.  

Since $\del_t \theta_N$ is uniformly bounded in $L^2(H^{-1})$, we know $\theta_N(t,\cdot) \to \theta_0$ weakly in $L^2$ uniformly in $N$ and so the same holds for $\theta$.  
\vskip.3cm

Lastly, for any constant $a \geq 0$, the function $(\theta - a)_+$ satisfies
\begin{align*}
\ddt \int \frac{1}{2} (\theta - a)_+^2 + \int \abs{\Lambda^{1/2} (\theta - a)_+}^2 &= \frac{-1}{2}\int u \cdot \grad (\theta - a)_+^2 - \int a \Lambda (\theta - a)_+ - \eps \int \abs{\grad (\theta - a)_+}^2
\\ &= - \int a (\theta - a)_+ B_1 - \eps \int \abs{\grad (\theta - a)_+}^2
\\ &\leq 0. 
\end{align*}
This inequality is scaling-invariant, so the same holds for $\lambda \theta(\mu t, \mu x)$ for any $\lambda, \mu > 0$.  

By the standard De Giorgi argument (see Lemma~\ref{thm:DG1 skeleton} in the Appendix for details), there exists a universal constant $\delta$ such that $\int_0^2 \int (\lambda \theta(\mu t, \mu x)_+^2 \,dxdt \leq \delta$ implies $\theta \leq \lambda\n$ on $[\mu 1,\mu 2]$.  In fact, by comparison with a constant super-solution, $\theta \leq \lambda\n$ on $[\mu, \infty)$.  Taking $\lambda = \sqrt{\frac{\delta}{2 \mu^{-2} \norm{\theta_0}_{L^2(\Omega)}^2}}$, we find $\theta(t,\cdot) \leq C t^{-1} \norm{\theta_0}_{L^2(\Omega)}$ for a universal constant $C$.  Applying the same argument to $-\theta$ gives the $L^\infty$ bound.  
\end{proof}

Now that we have global existence of solutions to \eqref{eq:viscous} for $\eps > 0$, we can prove Proposition~\ref{thm:existence} by taking a limit as $\eps \to 0$.  

\begin{proof}[Proof of Proposition \ref{thm:existence}]
For any parameter $\eps > 0$, define $\theta_\eps \in L^2(H_0^1)$ the weak solution to \eqref{eq:viscous} constructed in Lemma~\ref{thm:viscous}.  The $\theta_\eps$ are uniformly bounded in $L^\infty(L^2)$ and $L^2(\HD^{1/2})$ by the standard energy argument, so by interpolation they are also uniformly bounded in $L^4(L^{8/3})$.  Recall that $\theta_\eps$ are uniformly bounded in $L^\infty(L^\infty)$ after any positive time.  

For any smooth $\varphi$, we have
\[ \int_0^\infty \int \del_t \theta_\eps \varphi \leq \norm{\theta_\eps}_{L^4(L^{8/3})}^2 \norm{\varphi}_{L^2(W^{1,4})} + \norm{\theta_\eps}_{L^2(\HD^{1/2})} \norm{\varphi}_{L^2(\HD^{1/2})} + \eps \norm{\theta_\eps}_{L^2(\HD^{1/2})} \norm{\varphi}_{L^2(\HD^{3/2})}. \]
Therefore $\del_t \theta_\eps$ is uniformly bounded in $L^2(\HD^{-3/2})$.  By Aubin-Lions, the sequence $\theta_\eps$, up to a subsequence, has a strong limit in $L^2(L^2)$.  Call this limit $\theta$.  

Since $\del_t \theta_\eps$ is uniformly bounded and $\theta(t,\cdot) \to \theta_0$ weakly in $L^2$, the same holds for $\theta$.  


Define $u_\eps := \grad^\perp \Lambda^{-1} \theta_\eps$, and by continuity of the Riesz transform we have $u_\eps \to u$ in $L^2([0,\infty)\times\Omega)$ where $u := \grad^\perp\Lambda^{-1} \theta$.  
\vskip.3cm

It remains only to prove that $\theta$ and $u$ satisfy the energy inequalities \eqref{leray-hopf condition} and \eqref{local condition}.  Recall that $\theta_\eps$ and $u_\eps$ satisfy \eqref{leray-hopf condition} and \eqref{local condition} by Proposition~\ref{thm:suitability}, so we need only show that these inequalities hold also in the limit.  The details of this calculation are given below.  

Let $0 < T < \infty$ be a constant, and let $\lambda$, $\mu$, $\Psi$, $S$ and $\varphi$ be as in the definition of suitable solutions.  Define 
\newcommand{\thetae}{\tilde{\theta}_\eps}
\newcommand{\ue}{\tilde{u}_\eps}
\newcommand{\thetaep}{\tilde{\theta}_{\eps,+}}
\begin{align*}
\thetae(t,x) &:= \lambda\theta_\eps(\mu t, \mu x), & \tilde{\theta}(t,x) &=  \lambda \theta(\mu t, \mu x), \\
\ue(t,x) &:= u_\eps(\mu t, \mu x), & \tilde{u}(t,x) &:= u(\mu t, \mu x), \\
\thetaep(t,x) &:= \paren{\thetae(t, x) - \Psi(t,x)}_+, & \tilde{\theta}_+(t,x) &:= \paren{\tilde{\theta}(t, x) - \Psi(t,x)}_+,
\end{align*}
and let $\tilde{\Omega} := \{x \in \R^2: \mu x \in \Omega\}$, $\tilde{T} := \mu\n T$, $\tilde{\eps} := \mu\n \eps$.  

Note that $\thetae$ and $\ue$ are weak solutions to \eqref{eq:viscous linear} with viscosity $\tilde{\eps}$.  Therefore, if $\eps \leq \mu$ then $\thetaep$ and $\ue$ satisfy \eqref{leray-hopf condition}.  The terms $\int \thetaep (\del_t\Psi+\ue\cdot\grad\Psi)$ and $\int \indic{\thetae \geq 0}$ are continuous under $L^2$ limits, and the quantities $\ddt \int \thetaep^2$ and $\int \abs{\Lambda^{1/2} \thetaep}^2$ are lower-semicontinuous under $L^2$ limits, so we conclude that $\tilde{\theta}_+$ and $\tilde{u}$ satisfy \eqref{leray-hopf condition}.  

Similarly, $\thetaep$ and $\ue$ satisfy \eqref{local condition} if $\eps \leq \mu$.  On $[S, \tilde{T}]$ we have a uniform $L^\infty(L^\infty)$ bound for $\thetae$. Therefore $\thetaep$ converges in $L^3(L^3)$, and so $\int_S^{\tilde{T}}\int \thetaep^2 \ue\cdot\grad\varphi$ is conserved in the limit $\eps \to 0$.  The remaining terms in \eqref{local condition} are $L^2(L^2)$ continuous, so $\tilde{\theta}_+$ and $\tilde{u}$ satisfy \eqref{local condition}.  
\end{proof}

\vskip1cm
\section{Littlewood-Paley Theory} \label{sec:littlewood paley}

In this section we will prove that, because $\theta$ is uniformly bounded in $L^\infty$, the velocity $u = \grad^\perp \Lambda^{-1} \theta$ is calibrated (see Definition~\ref{def:calibrated}).  The proof will utilize a Littlewood-Paley theory adapted to a bounded set $\Omega$.  

Because the Littlewood-Paley theory depends in an essential way on the domain $\Omega$, any results proven in this way will also be domain-dependent.  Therefore, in the proof of H\"{o}lder continuity in Section~\ref{sec:holder}, we will apply the following Proposition only to the unscaled function $\theta$ on the unscaled domain $\Omega$.  As we zoom in, the velocity will remain calibrated, so there will be no further need for this result.  

\begin{proposition} \label{thm:u is calibrated}
Let $\Omega \subseteq \R^2$ be a bounded set with $C^{2,\beta}$ boundary for some $\beta \in (0,1)$.  Let $\theta \in L^\infty(\Omega)$.  Then there exists an integer $j_0 = j_0(\Omega)$ and a sequence of divergence-free functions $(u_j)_{j \geq j_0}$ calibrated for some constant $\Ccalib = \Ccalib(\Omega, \norm{\theta}_\infty)$ with center 0 (see Definition~\ref{def:calibrated}) such that
\[ \grad^\perp \Lambda^{-1} \theta = \sum_{j \geq j_0} u_j \]
with the infinite sum converging in the sense of $L^2$.  
\end{proposition}

Before we can prove this, we define the Littlewood-Paley projections and prove some of their properties:

Let $\phi$ be a Schwartz function on $\R$ which is suited to Littlewood-Paley decomposition.  Specifically, $\phi$ is non-negative, supported on $[1/2,2]$, and has the property that
\[ \sum_{j \in \Z} \phi(2^j \xi) = 1 \qquad \forall \xi \neq 0. \]
For any $f = \sum f_k \eigen{k}$ in $L^2(\Omega)$, we define the Littlewood-Paley projections
\[ P_j f := \sum_k \phi(2^j \lambda_k^{1/2}) f_k \eigen{k}. \]
Note that $P_j$ depends strongly on the domain $\Omega$.  

Recall that $-\Laplace_D$ has some smallest eigenvalue $\lambda_0$ (depending on $\Omega$) so if we define $j_0 = \log_2(\lambda_0)-1$ then $P_j = 0$ for all $j < j_0$.

The Bernstein Inequalities adapted for a bounded domain are proved in \cite{IMT.bilinear}.  We restate their result here:
\begin{lemma}[Bernstein Inequalities] \label{thm:IMT stuff}
Let $1 \leq p \leq \infty$ and $\Omega \subset \R^2$ a bounded open set with $C^{2,\beta}$ boundary for some $\beta \in (0,1)$, and let $(P_j)_{j \in \Z}$ be the Littlewood-Paley decomposition defined above.  

There exists a constant $C$ depending on $p$ and $\Omega$ such that the following hold for any $f \in L^p(\Omega)$:

For any $\alpha \in \R$ and $j \in \Z$, 
\[ \norm{\Lambda^\alpha P_j f}_{L^p(\Omega)} \leq C 2^{\alpha j} \norm{f}_{L^p(\Omega)}. \]

For any $\alpha \in \R$ and $j \geq j_0$
\[ \norm{\grad \Lambda^\alpha P_j f}_{L^p(\Omega)} \leq C 2^{(1+\alpha) j} \norm{f}_{L^p(\Omega)}. \]
\end{lemma}

\begin{proof}
The first claim is Lemma 3.5 in \cite{IMT.bilinear}.  It is also an immediate corollary of \cite{IMT.schrodinger} Theorem 1.1.  

The second claim is similar to Lemma 3.6 in \cite{IMT.bilinear}.  A hypothesis of Lemma 3.6 is that
\[ \norm{\grad e^{-t\Laplace_D}}_{L^\infty \to L^\infty} \leq \frac{C}{\sqrt{t}} \qquad 0 < t \leq 1 \]
(a property of $\Omega$).  The result of Lemma 3.6 only covers the case $j > 0$.  

In \cite{FMP} it is proved that that if $\Omega$ is $C^{2,\beta}$ then
\[ \norm{\grad e^{-t\Laplace_D}}_{L^\infty \to L^\infty} \leq \frac{C}{\sqrt{t}} \qquad 0 < t \leq T \]
which, by taking some $T$ depending on $j_0$, is enough to prove the desired result for $j \geq j_0$ by a trivial modification of the proof in \cite{IMT.bilinear}.  
\end{proof}

The following lemma is a simple but crucial result which can be thought of as describing the commutator of the gradient operator and the projection operators.  In the case of $\R^2$, the Littlewood-Paley projections commute with the gradient so $P_i \grad P_j = 0$ unless $|i-j|\leq 1$.  On a bounded domain, this is not the case; the gradient does not maintain localization in frequency-space.  However, the following lemma formalizes the observation that $P_i \grad P_j \approx 0$ when $i << j$.  

\begin{lemma} \label{thm:grad and proj}
Let $1 \leq p \leq \infty$.  There exists a constant $C$ depending on $p$ and $\Omega$ such that or any function $f \in L^p(\Omega)$,
\[ \norm{P_i \grad P_j f}_p \leq C \min(2^j,2^i) \norm{f}_p. \]
\end{lemma}
\begin{proof}
Let $q$ be the H\"{o}lder conjugate of $p$ and $g$ be an $L^q$ function.  Then since $P_i$ is self-adjoint
\[ \int g P_i \grad P_j f = \int (P_i g) \grad P_j f \leq C 2^j \norm{g}_q \norm{f}_p \]
by Lemma \ref{thm:IMT stuff}.  

Further integrating by parts,
\[ \int g P_i \grad P_j f = - \int (\grad P_i g) P_j f \leq C 2^i \norm{g}_q \norm{f}_p. \]
This also follows from Lemma \ref{thm:IMT stuff}.  

The result follows.  
\end{proof}

We are now ready to prove Proposition \ref{thm:u is calibrated}.  

\begin{proof}[Proof of \ref{thm:u is calibrated}]
For each integer $j \geq j_0$, we define $u_j$ to be the $\frac{\pi}{2}$-rotation of the Riesz transform of the $j\ith$ Littlewood-Paley projection of $\theta$:
\[ u_j := \grad^\perp \Lambda^{-1} P_j \theta. \]
Qualitatively, we know that $\theta \in L^2$ and hence $u_j \in L^2$.  In fact, $u = \sum u_j$ in the $L^2$ sense.  

We must bound $u_j$, $\Lambda^{-1/4} u_j$, and $\grad u_j$ all in $L^\infty(\Omega)$.  
\vskip.3cm

By straightforward application of Lemma~\ref{thm:IMT stuff},
\begin{equation} \label{uj in Linfty} \norm{u_j}_\infty \leq C \norm{\theta}_\infty. \end{equation}
\vskip.3cm

Since $u_j \in L^2$, we know that
\[ \Lambda^{-1/4} u_j = \sum_{i \in \Z} P_i \Lambda^{-1/4} u_j. \]
Define $\bar{P}_k := P_{k-1} + P_k + P_{k+1}$.  Then $\bar{P}_k P_k = P_k$, and since the projections $P_k$ are spectral operators, they commute with $\Lambda^s$ and each other.  We therefore rewrite
\[ \paren{P_i \Lambda^{-1/4} u_j}^\perp = \paren{\Lambda^{-1/4} \bar{P}_i} \paren{P_i \grad P_j} \paren{\Lambda^{-1} \bar{P}_j} \theta. \]
On the right hand side we have three bounded linear operators applied sequentially to $\theta \in L^\infty$.  The first operator has norm $C 2^{-j}(2^{1}+2^0+2^{-1})$ by Lemma~\ref{thm:IMT stuff}.  The second operator has norm $C \min(2^j, 2^i)$ by Lemma~\ref{thm:grad and proj}.  The third operator has norm $C 2^{-i/4}(2^{1/4} + 2^0 + 2^{-1/4})$ by Lemma~\ref{thm:IMT stuff}.  Therefore
\[ \norm{ P_i \Lambda^{-1/4} u_j}_\infty \leq C 2^{-i/4} \min(2^j, 2^i) 2^{-j} \norm{\theta}_\infty. \]

Summing these bounds on the projections of $\Lambda^{-1/4} u_j$, and noting that
\[ \sum_{i \in \Z} 2^{-j} 2^{-i/4} \min(2^j,2^i) = 2^{-j} \sum_{i \leq j} 2^{i 3/4} + \sum_{i>j} 2^{-i/4} \leq C 2^{-j/4}, \]
we obtain
\begin{equation}\label{uj in W-1/4} \norm{\Lambda^{-1/4} u_j}_\infty \leq C 2^{-j/4} \norm{\theta}_\infty. \end{equation}
\vskip.3cm

Lastly, we must show that $\grad u_j$ is in $L^\infty$.  Equivalently, we will show that $\Lambda^{-1} P_j \theta$ is $C^{1,1}$.  The method of proof is Schauder theory.  

For convenience, define
\[ F := \Lambda^{-1} P_j \theta. \]
Notice that $F$ is a linear combination of Dirichlet eigenfunctions, so in particular it is smooth and vanishes at the boundary.  Therefore
\[ -\Laplace F = \Lambda^2 F = \Lambda P_j \theta. \]


We apply the standard Schauder estimate from Gilbarg and Trudinger \cite{GiTr} Theorem 6.6 to bound some $C^{2,\alpha}$ semi-norm of $F$ by the $L^\infty$ norm of $F$ and the $C^\alpha$ norm of its Laplacian.  By assumption there exists $\beta \in (0,1)$ such that $\Omega$ is $C^{2,\beta}$, and for this $\beta$ we have by the Schauder estimate
\begin{equation} \label{schauder estimate} \bracket{D^2 F}_{\beta} \leq C \norm{\Lambda^{-1} P_j \theta}_\infty + C \norm{\Lambda P_j \theta}_\infty + C \bracket{\Lambda P_j \theta}_{\beta}. \end{equation}

By Lemma \ref{thm:IMT stuff},
\begin{align*} 
\norm{\Lambda^{-1} P_j \theta}_\infty &\leq C 2^{-j} \norm{\theta}_\infty, \\
\norm{\Lambda P_j \theta}_\infty &\leq C 2^j \norm{\theta}_\infty, \\
\norm{\grad \Lambda P_j \theta}_\infty &\leq C 2^{2j} \norm{\theta}_\infty. 
\end{align*}
By Lemma~\ref{thm:interpolation C1 to holder} (see Appendix~\ref{sec:technical}) we can interpolate these last two bounds to obtain
\[ \bracket{\Lambda P_j \theta}_\beta \leq C 2^{j(1+\beta)} \norm{\theta}_\infty. \]
Plugging these estimates into \eqref{schauder estimate} yields
\[ \bracket{D^2 F}_\beta \leq C \paren{2^{-j} + 2^j + 2^{j(1+\beta)}} \norm{\theta}_\infty. \]

Recall that without loss of generality we can assume $j \geq j_0$.  Therefore up to a constant depending on $j_0$, the term $2^{j(1+\beta)}$ bounds $2^j$ and $2^{-j}$ so we can write
\[ \bracket{D^2 F}_\beta \leq C 2^{j(1+\beta)} \norm{\theta}_\infty. \]

Using this estimate and the fact that $\norm{\grad F}_\infty = \norm{\grad \Lambda^{-1} P_j \theta}_\infty \leq C \norm{\theta}_\infty$ (see \eqref{uj in Linfty}), we can interpolate to obtain an $L^\infty$ bound on $D^2 F$.  
Lemma~\ref{thm:interpolation holder to C1} states that since $F \in C^{2,\beta}$ and $\Omega$ is sufficiently regular, there exist a constant $\ell = \ell(\Omega)$ such that for any $\delta \in [0, \ell]$ we have
\begin{align*}
\norm{D^2 F}_\infty &\leq C \paren{\delta\n \norm{\grad F}_\infty + \delta^{\beta} \bracket{D^2 F}_\beta}
\\ &\leq C \paren{\delta\n + \delta^\beta 2^{j(1+\beta)}} \norm{\theta}_\infty.  
\end{align*}

Set $\delta = 2^{-j} (2^{j_0} \ell) \leq \ell$.  Then
\[ \norm{D^2 F}_\infty \leq C \paren{2^j + 2^{-j\beta} 2^{j(1+\beta)}} \norm{\theta}_\infty = C(\Omega) 2^j \norm{\theta}_\infty. \]

Since $D^2 F = \grad u_j$, this estimate together with \eqref{uj in Linfty} and \eqref{uj in W-1/4} complete the proof.  
\end{proof}

\vskip1cm
\section{De Giorgi Estimates} \label{sec:de giorgi}

Our goal in this section is to prove De Giorgi's first and second lemmas for suitable solutions to \eqref{eq:main linear} with $u$ uniformly calibrated.  The De Giorgi lemmas will eventually be applied iteratively to various rescalings of the solution $\theta$, so the following results must be independent of the size of the domain $\Omega$.  Any properties we do assume for the domain, such as the regularity of the boundary, must be scaling invariant.  
\vskip.3cm

Rather than working directly with the calibrated sequence, we will decompose $u$ into just two terms, a low-pass term and a high-pass term.  The construction is described in the following lemma.  Note that we make no assumption on the center of calibration, which means this result is indendent of scale.  
\begin{lemma} \label{thm:calibration is good}
Let 
\[ u = \sum_{j_0}^\infty u_j \]
with the sum converging in the $L^2$ sense.  Assume that $(u_j)_{j \in \Z}$ is a calibrated sequence with constant $\Ccalib$ and some center, and that $\div(u_j)=0$ for all $j$.

Then
\[ u = \ulow + \uhigh \]
with 
\begin{align*} 
\norm{\grad \ulow}_{L^\infty([-T,0]\times \Omega)} &\leq 2 \Ccalib, \\
\norm{\Lambda^{-1/4} \uhigh}_{L^\infty([-T,0]\times\Omega)} &\leq 6 \Ccalib. 
\end{align*}
and $\div(\ulow) = \div(\uhigh) = 0$.  
\end{lemma}

We call $\ulow$ the low-pass term, and $\uhigh$ the high-pass term.  

\begin{proof}
Let $N$ be the center to which $(u_j)_{j \in \Z}$ is calibrated.  

We define
\[ \uhigh = \sum_{j = N+1}^\infty u_j \]
and bound
\[ \norm{\Lambda^{-1/4} \uhigh}_\infty \leq \sum_{j>N} \norm{\Lambda^{-1/4} u_j}_\infty \leq \Ccalib \frac{2^{-1/4}}{1-2^{-1/4}}. \]

We define
\[ \ulow = \sum_{j = j_0}^N u_j \]
and bound
\[ \norm{\grad \ulow}_\infty \leq \sum_{j \leq N} \norm{\grad u_j}_\infty \leq \Ccalib \frac{1}{1 - 2^{-1}}. \]
\end{proof}

In order to prove the De Giorgi lemmas, we must derive an energy inequality for the function $\paren{\theta - \Psi}_+$ where $\Psi(t,x)$ grows sublinearly in $|x|$.  Considering the suitability condition \eqref{leray-hopf condition}, we see that control can only be gained if the quantity $\del_t \Psi + u \cdot \grad\Psi$ is bounded.  This requires a barrier function which is moving in space along a Lagrangian path $\Gammal$ of $\ulow$.  
\vskip.3cm

To that end, we shall consider, for any domain $\Omega$ and time $T$, functions $\theta:[-T,0]\times\Omega \to \R$, $L^2$ functions $\ulow$ and $\uhigh:[-T,0]\times \Omega \to \R^2$, and a Lipschitz path $\Gammal:[-T,0] \to \Omega$ which satisfy
\begin{equation} \label{eq:main linear brokendown} \begin{cases}
\theta, (\ulow + \uhigh) \textrm{ suitable solution to \eqref{eq:main linear}} & \textrm{on } [-T,0] \times \Omega, \\
\div( \ulow) = \div(\uhigh) = 0 & \textrm{on } [-T,0]\times \Omega, \\
\dot{\Gammal}(t) = \ulow(t, \Gammal(t)) & \textrm{on } [-T,0].
\end{cases} \end{equation}

Because $\Gammal$ depends on $\ulow$ which depends on $N$, the path $\Gammal$ will change significantly between scales.  In particular, though $\Gammal \in \Lip([-T,0];\R^2)$, we cannot assume any uniform bound on it Lipschitz constant.  We can bound, however, the difference between $\Gammal$ at consecutive scales.  Therefore we must consider in the following lemmas an arbitrary Lipschitz path $\Gamma$, which was produced at a previous scale, and denote $\gamma := \Gammal - \Gamma$ which will be uniformly bounded.  
\vskip.3cm

Now we prove an energy inequality for solutions to \eqref{eq:main linear brokendown}.  Though this lemma is independent of the size of the domain, it depends on the geometry of the domain in a way encoded by the constant $\Comega$.  We will later show that this constraint on $\Omega$ is scaling invariant.  

\begin{lemma}[Energy inequality] \label{thm:energy inequality}
Let $\Ccalib$, $\Comega$, $\Cgamma$, $T$, and $R$ be positive constants, and let $\psi:\R^2 \to \R$ be a function such that $\norm{\grad\psi}_\infty$, $\norm{D^2\psi}_\infty$, and $\sup_t \bracket{\psi(t,\cdot)}_{1/4}$ are all finite. Then there exists a constant $C>0$ such that the following holds:

Let $\Omega \subseteq \R^2$ be a bounded open set with $C^{2,\beta}$ boundary for some $\beta \in (0,1)$, and let $\Gamma:[-T,0] \to \R^2$ be Lipschitz.  Assume that on $\Omega$ the functions $K_{1/4}$ and $K_1$ (defined in \eqref{Caff Stinga representation}) satisfy the relation
\[ K_{1/4}(x,y) \leq \Comega |x-y|^{3/4} K_1(x,y) \qquad \forall x\neq y \in \Omega. \]

Let $\theta$, $\ulow$, $\uhigh$, and $\Gammal$ solve \eqref{eq:main linear brokendown} on $[-T,0]\times\Omega$, and satisfy $\norm{\Lambda^{-1/4} \uhigh}_{L^\infty([-T,0]\times\Omega)} \leq 6 \Ccalib$ and $\norm{\grad \ulow}_{L^\infty([-T,0]\times\Omega)} \leq 2\Ccalib$.   Denote $\gamma := \Gammal - \Gamma$ and assume $\norm{\dot{\gamma}}_{L^\infty([-T,0])} \leq \Cgamma$ and $\gamma(0) = 0$.  
\vskip.3cm

Consider the functions
\[ \theta_+ := \paren{\theta - \psi(\cdot-\Gamma)}_+, \qquad \theta_- := \paren{\psi(\cdot-\Gamma) - \theta}_+. \]

If $\theta_+$ is supported on $x \in \Omega \cap B_R(\Gamma(t))$ then $\theta_+$ and $\theta_-$ satisfy the inequality
\[ \ddt \int \theta_+^2 + \int \abs{\Lambda^{1/2} \theta_+}^2 - \int \Lambda^{1/2} \theta_+ \Lambda^{1/2} \theta_- \leq C \paren{ \int \indic{\theta \geq \psi} + \int \theta_+ + \int \theta_+^2 }. \]
\end{lemma}

\begin{proof}
Define 
\[ \Psi(t,x) := \psi(x - \Gamma(t)) \]
so that
\[ \del_t \Psi + (\ulow + \uhigh)\cdot\grad \Psi = (\ulow - \dot{\Gamma} + \uhigh)\cdot \grad \psi(x-\Gamma(t)). \]

Applying \eqref{leray-hopf condition} we arrive at
\begin{equation} \label{control from cacciopolli} \ddt \int \theta_+^2 + \int \abs{\Lambda^{1/2} \theta_+}^2 \leq C \paren{ \int \indic{\theta \geq \psi} + \abs{\int \theta_+ (\ulow-\dot{\Gamma}(t) + \uhigh) \cdot \grad\psi(x-\Gamma(t))} }. \end{equation}
\vskip.3cm

Consider first the high-pass term $\int \theta_+ \uhigh\cdot\grad\psi$.  This term is equal to $\int \Lambda^{1/4} (\theta_+\grad\Psi) \Lambda^{-1/4} \uhigh$, as can be calculated by first decomposing $\theta_+\grad\Psi$ and $\uhigh$ as sums of eigenfunctions.  The operations on these infinite sums are justified because $\theta_+ \grad\Psi$, $\Lambda^{1/4} (\theta_+ \grad\Psi)$, $\uhigh$, and $\Lambda^{-1/4} \uhigh$ are all in $L^2$.  Therefore we can apply Lemma~\ref{thm:Lambda stuff} parts \eqref{thm:L1 of Lambda1/4 bounded} and \eqref{thm:extra product rule} to obtain
\[ \int \Lambda^{-1/4} \uhigh \Lambda^{1/4} (\theta_+ \grad\psi) \leq C \norm{\Lambda^{-1/4} \uhigh}_\infty \paren{\norm{\grad\psi}_\infty + \norm{D^2 \psi}_\infty} |\supp(\theta_+)|^{1/2} \paren{ \norm{\theta_+}_{L^2} + \norm{\theta_+}_{\HD^{1/2}}}. \]
We apply Young's inequality to find that for any constant $\eps > 0$ there exists $C = C(\psi,\Ccalib,\Comega,\eps)$ such that
\begin{equation} \label{low pass control} \int \uhigh \theta_+ \grad \psi(x - \Gamma(t)) \,dx \leq C \paren{|\supp(\theta_+)| + \int \theta_+^2} + \eps \int \abs{\Lambda^{1/2} \theta_+}^2. \end{equation}
\vskip.3cm

Consider now the low-pass term.  By \eqref{eq:main linear brokendown}
\begin{equation} \label{low pass term definition} \ulow(t,x) - \dot{\Gamma}(t) = \ulow(t,x) - \ulow(t,\Gamma+\gamma) + \dot{\gamma}. \end{equation}

Since $\ulow$ is has derivative bounded by $2\Ccalib$,
\begin{align*} 
|\ulow(t,x) - \ulow(t,\Gamma+\gamma)| &\leq  |\ulow(t,x)-\ulow(t,\Gamma)| + |\ulow(t,\Gamma) - \ulow(t,\Gamma+\gamma)| 
\\ &\leq 2\Ccalib |x-\Gamma| + 2\Ccalib |\gamma|. 
\end{align*}
By assumption $|\dot{\gamma}|\leq \Cgamma$ and $\gamma(0)=0$, and so for $t \in [-T,0]$ we have $|\gamma(t)| \leq T \Cgamma$.  

Plugging these bounds into \eqref{low pass term definition} we obtain
\[ \abs{\ulow(t,x) - \dot{\Gamma}(t)} \leq 2\Ccalib |x-\Gamma| + 2\Ccalib T \Cgamma + \Cgamma. \]

Now we can bound the low pass term
\[ \int (\ulow - \dot{\Gamma}) \theta_+ \grad\psi(x-\Gamma) \leq (2\Ccalib T+1) \Cgamma \norm{\grad\psi}_\infty \int \theta_+ \,dx +  \norm{\grad\psi}_\infty 2\Ccalib \int |x-\Gamma| \theta_+ \,dx. \]
By assumption, $|x-\Gamma| \theta_+ \leq R \theta_+$, so from this, \eqref{low pass control}, and \eqref{control from cacciopolli} the result follows.  
\end{proof}

This energy inequality is sufficient to prove the De Giorgi Lemmas.  
\vskip.3cm

The first lemma is a local version of the $L^2$ to $L^\infty$ regularization, stating that solutions with small $L^2$ norm in a region will have small $L^\infty$ norm in a smaller region.  

\begin{proposition}[First De Giorgi Lemma] \label{thm:DG1}

Let $\Ccalib$, $\Comega$, and $\Cgamma$, be positive constants. Then there exists a constant $\delta_0>0$ such that the following holds:

Let $\Omega \subseteq \R^2$ be a bounded open set with $C^{2,\beta}$ boundary for some $\beta \in (0,1)$, and let $\Gamma:[-2,0] \to \R^2$ be Lipschitz.  Assume that on $\Omega$ the functions $K_{1/4}$ and $K_1$ (defined in \eqref{Caff Stinga representation}) satisfy the relation
\[ K_{1/4}(x,y) \leq \Comega |x-y|^{3/4} K_1(x,y) \qquad \forall x\neq y \in \Omega. \]

Let $\theta$, $\ulow$, $\uhigh$, and $\Gammal$ solve \eqref{eq:main linear brokendown} on $[-2,0]\times\Omega$, and satisfy $\norm{\Lambda^{-1/4} \uhigh}_{L^\infty([-2,0]\times\Omega)} \leq 6 \Ccalib$ and $\norm{\grad \ulow}_{L^\infty([-2,0]\times\Omega)} \leq 2\Ccalib$.  Denote $\gamma := \Gammal-\Gamma$ and assume $\norm{\dot{\gamma}}_{L^\infty([-2,0])} \leq \Cgamma$ and $\gamma(0) = 0$.  
\vskip.3cm

If
\[ \theta(t,x) \leq 2 + \paren{|x-\Gamma(t)|^{1/4}-2^{1/4}}_+ \qquad \forall t\in[-2,0], x \in \Omega \setminus B_2(\Gamma(t)) \]
and
\[ \int_{-2}^0 \int_{\Omega\cap B_2(\Gamma(t))} (\theta)_+^2 \,dxdt \leq \delta_0 \]
then
\[ \theta(t,x) \leq 1 \qquad \forall t \in [-1,0], x \in \Omega \cap B_1(\Gamma(t)). \]

\end{proposition}

\begin{proof}
Let $\psi$ be such that $\psi = 0$ for $|x| \leq 1$ and $\psi(x) = 2 + \paren{|x|^{1/4}-2^{1/4}}_+$ for $|x|>2$, and let $\grad \psi$ and $D^2 \psi$ be bounded.  

For any constant $a > 0$, we can apply Lemma~\ref{thm:energy inequality} to the function
\[ \theta_a := (\theta(t,x) - \psi(x - \Gamma(t)) - a)_+ \]
and obtain
\[ \ddt \int \theta_a^2 + \int \abs{\Lambda^{1/2} \theta_a}^2 \leq C \paren{ \int \indic{\theta \geq \psi + a} + \int \theta_a + \int \theta_a^2 }. \]

Thus $\theta - \psi(x-\Gamma)$ satisfies the assumptions of Lemma~\ref{thm:DG1 skeleton}.  There exists a constant, which we call $\delta_0$, so that if
\[ \int_{-2}^0 \int \paren{\theta(t,x) - \psi(x-\Gamma(t))}_+ \,dxdt \leq \delta_0 \]
then 
\[ \theta(t,x) \leq 1 + \psi(x-\Gamma(t)) \qquad \forall t \in [-1,0], x \in \Omega. \]

By construction of $\psi$, our result follows immediately.  

\end{proof}

Next, we will prove De Giorgi's second lemma, a quantitative analog of the isoperimetric inequality.  

\begin{proposition}[Second De Giorgi Lemma] \label{thm:DG2}
Let $\Ccalib$, $\Comega$, $\Cgamma$, and $\beta \in (0,1)$ be positive constants. Then there exists a constant $\mu>0$ such that the following holds:

Let $\Omega \subseteq \R^2$ be a bounded open set with $C^{2,\beta}$ boundary for some $\beta \in (0,1)$ , and let $\Gamma:[-5,0] \to \R^2$ be Lipschitz.  Assume that on $\Omega$ the functions $K_{1/4}$ and $K_1$ (defined in \eqref{Caff Stinga representation}) satisfy the relation
\[ K_{1/4}(x,y) \leq \Comega |x-y|^{3/4} K_1(x,y) \qquad \forall x\neq y \in \Omega. \]

Let $\theta$, $\ulow$, $\uhigh$, and $\Gammal$ solve \eqref{eq:main linear brokendown} on $[-5,0]\times\Omega$, and satisfy $\norm{\Lambda^{-1/4} \uhigh}_{L^\infty([-5,0]\times\Omega)} \leq 6 \Ccalib$ and $\norm{\grad \ulow}_{L^\infty([-5,0]\times\Omega)} \leq 2\Ccalib$.  Denote $\gamma := \Gammal - \Gamma$ and assume $\norm{\dot{\gamma}}_{L^\infty([-5,0])} \leq \Cgamma$ and $\gamma(0) = 0$.  
\vskip.3cm

Suppose that for $t \in [-5,0]$ and any $x \in \Omega$,
\[ \theta(t,x) \leq 2 + \paren{|x-\Gamma(t)|^{1/4}-2^{1/4}}_+. \]

Then the three conditions
\begin{align} 
\abs{\{\theta \geq 1\} \cap [-2,0]\times B_2(\Gamma)} &\geq \delta_0/4, \label{mass assumption above} \\
\abs{\{0 < \theta < 1\} \cap [-4,0]\times B_4(\Gamma)} &\leq \mu, \nonumber \\
\abs{\{\theta \leq 0\} \cap [-4,0]\times B_4(\Gamma)} &\geq 2 |B_4| \label{mass assumption below}
\end{align}
cannot simultaneously be met.  
\end{proposition}

Here $\delta_0$ is the constant from Proposition~\ref{thm:DG1}, which of course depends on $\Ccalib$, $\Cgamma$, and $\Comega$.  

\begin{proof}
Suppose that the proposition is false.  Then there must exist, for each $n \in \N$, a bounded open set $\Omega_n$ with $C^{2,\beta_n}$ boundary for $\beta_n \in (0,1)$, a Lipschitz path $\Gamma_n:[-5,0]\to \R^2$, a function $\theta_n: [-5,0]\times \Omega_n \to \R$, functions $\ulow^n, \uhigh^n: [-5,0]\times\Omega_n \to \R^2$, and paths $\Gammal^n = \Gamma_n + \gamma_n:[-5,0]\to\R^2$ which solve \eqref{eq:main linear brokendown} and satisfy all of the the assumptions of our proposition (with the same constants $\Ccalib$, $\Cgamma$, and $\Comega$), except that
\begin{equation} \label{mass assumption between} \abs{\{0 < \theta_n < 1\} \cap [-4,0]\times B_4(\Gamma_n)} \leq 1/n. \end{equation}
\vskip.3cm

Let $\psi:\R^2 \to \R$ be a smooth function which vanishes on $B_2$ such that $\psi(x) = 2 + \paren{|x|^{1/4}-2^{1/4}}_+$ for $|x|>3$.  

Fix $n$ and define 
\[ \theta_+ := \paren{\theta_n - \psi(x-\Gamma_n)}_+. \]
Then $\theta_+$ is supported on $\Omega \cap B_3(\Gamma_n)$ and is less than $2 + 3^{1/4} - 2^{1/4} \leq 3$ everywhere.  

Our goal is to bound the derivatives of $\theta_+^2$ so that we can apply a compactness argument to the sequence $\theta_n$.  (It is the calculations in Step 2 below in which it becomes necessary to consider $\theta_+^2$ instead of $\theta_+$.)  

The remainder of the proof is divided in three steps.  First we show that the sequence of $\theta_+$ is compact in space, then we show that it is compact in time, and finally we show that the limiting function implies a contradiction.  

\vskip.3cm
\step{Compactness in space}

Apply the energy inequality Lemma \ref{thm:energy inequality} to $\theta$ and $\psi(x-\Gamma_n)$, and find that for some $C$ independent of $n$
\begin{equation} \label{ddt theta bounded} \ddt \int \theta_+^2 \leq C. \end{equation}
Moreover, by integrating Lemma~\ref{thm:energy inequality} in time from $-5$ to $s \in [-4,0]$ and taking a supremum over $s$, we find
\begin{equation}\label{DG2 energy} \sup_{[-4,0]} \int \theta_+^2 + \int_{-4}^0 \int \abs{\Lambda^{1/2}\theta_+}^2 + \int_{-4}^0 \int \Lambda^{1/2}\theta_+ \Lambda^{1/2}\theta_- \leq C. \end{equation}
This proves in particular that $\theta_+ \in L^2(-4,0; \HD^{1/2}(\Omega))$ is uniformly bounded.  

Furthermore, $\norm{\theta_+^2}_{L^2(-4,0;\HD^{1/2}(\Omega_n))}$ is uniformly bounded because
\begin{align*} 
\norm{\Lambda^{1/2}(\theta_+^2)}_2^2 &= \iint [\theta_+(x)^2 - \theta_+(y)^2]^2 K + \int \theta_+^4 B 
\\ &\leq 2\iint \theta_+(x)^2 [\theta_+(x)-\theta_+(y)]^2 K + 2\iint \theta_+(y)^2[\theta_+(x)-\theta_+(y)]^2 K + \norm{\theta_+}_\infty^2 \int \theta_+^2 B
\\ &\leq C \norm{\theta_+}_\infty^2 \norm{\theta_+}_{\HD^{1/2}}^2.  
\end{align*}

By Proposition \ref{thm:hadamard 3 lines}, for $E$  the extension-by-zero operator from $L^2(\Omega_n)$ to $L^2(\R^2)$,
\begin{equation} \label{theta3 compact in space} \norm{ E \theta_+^2}_{L^2(-4,0; H^{1/2}(\R^2))} \leq C \end{equation}
where $C$ does not depend on $n$.  

\vskip.3cm
\step{Compactness in time}

Let $\varphi \in C_0^\infty([-4,0]; C^\infty(\Omega))$ a test function.  
Since each $\theta_n$ and $\uhigh^n + \ulow^n$ is a suitable solution to \eqref{eq:main linear} on $[-5,0]\times\Omega_n$ by assumption, we can apply the inequality \eqref{local condition} to find that, for some constant $C$ independent of $n$ and of $\varphi$, on $[-4,0]\times \Omega_n$
\begin{equation}\label{time compactness setup}
\begin{aligned} 
\iint \varphi \del_t \theta_+^2 + \iint \varphi \dot{\Gamma}_n\cdot\grad\theta_+^2 &\leq 
\iint \theta_+^2 \paren{\ulow^n - \dot{\Gamma}_n + \uhigh^n}\cdot\grad\varphi 
- 2 \iint \varphi \theta_+ \paren{\ulow^n - \dot{\Gamma}_n + \uhigh^n} \cdot \grad\psi 
\\ & + C \norm{\varphi}_{C^0(C^2)} \paren{ 1 + \int_{-5}^0 \abs{ \int \theta_+ \paren{\ulow^n - \dot{\Gamma}_n + \uhigh^n} \cdot \grad\psi} }.
\end{aligned}\end{equation}

For the low pass terms, as in the proof of Lemma~\ref{thm:energy inequality}, we have $\abs{\ulow^n(t,x) - \dot{\Gamma}_n(t)} \leq (1+8\kappa)\Cgamma + 6\kappa$ for $t \in [-4,0]$ and $x \in \supp(\theta_+) \subseteq B_3(\Gamma_n(t))$.  Thus for $t\in[-4,0]$ we have for $C$ independent of $n$ and $\varphi$
\begin{equation}\begin{aligned} \label{low pass bounds}
\int \paren{\ulow^n - \dot{\Gamma}_n}\cdot\paren{\theta_+^2\grad\varphi} & \leq C \norm{\grad\varphi}_{L^\infty(\Omega)}, \\
\int \paren{\ulow^n - \dot{\Gamma}_n}\cdot\paren{\theta_+\varphi\grad\psi} & \leq C \norm{\varphi}_{L^\infty(\Omega)}, \\
\int \paren{\ulow^n - \dot{\Gamma}_n}\cdot\paren{\theta_+\grad\psi} & \leq C.
\end{aligned} \end{equation}

For the high pass terms, we have $\uhigh^n$ uniformly bounded in $\dot{W}^{-1/4,\infty}$.  From step 1, we know $\theta_+^2$ is uniformly bounded in $L^2(-4,0;\HD^{1/2})$ so, by Lemma~\ref{thm:Lambda stuff} parts \ref{thm:L1 of Lambda1/4 bounded} and \ref{thm:extra product rule}, there is a constant $C$ independent of $n$ and $\varphi$ such that
\begin{equation}\begin{aligned} \label{high pass bounds}
\iint \uhigh^n \cdot (\theta_+^2 \grad\varphi) &\leq C \paren{\norm{\grad\varphi}_{C^0(-4,0;L^\infty(\Omega))} + \norm{\varphi}_{C^0(-4,0;C^2(\Omega))}}, \\
\iint \uhigh^n \cdot (\theta_+ \varphi\grad\psi) &\leq C \paren{\norm{\varphi}_{C^0(-4,0;L^\infty(\Omega))} + \norm{\varphi}_{C^0(-4,0;C^1(\Omega))}}, \\
\int_{-5}^0 \abs{\int \uhigh^n \cdot (\theta_+ \grad\psi)} &\leq C.
\end{aligned} \end{equation}

Plugging these six bounds into \eqref{time compactness setup}, for a constant $C$ independent of $n$ and $\varphi$, for any $\varphi$ nonnegative we have
\begin{equation} \label{theta3 compact in time} \int_{-4}^0 \int_{\Omega_n} \paren{\del_t \theta_+^2 + \dot{\Gamma}_n\cdot\grad\theta_+^2}\varphi \,dxdt \leq C \norm{\varphi}_{C^0(-4,0; C^2(\Omega_n))}. \end{equation}

Note that
\[ \int_{-4}^0 \int_{\Omega_n} \paren{\del_t \theta_+^2 + \dot{\Gamma}_n\cdot\grad\theta_+^2} \,dxdt = \theta_+(0,\Gamma_n(0))^2 - \theta_+(-4,\Gamma_n(-4))^2 \]
is uniformly bounded above and below.  Therefore, by decomposing $\varphi = (\varphi + \norm{\varphi}_{C^0}) - \norm{\varphi}_{C_0}$ into a non-negative smooth function plus a constant, we can see that \eqref{theta3 compact in time} holds for general $\varphi$.  

\vskip.3cm
\step{Taking the limit}

We wish to analyze the limiting behavior of $\theta_+^2$ in the vicinity of $\Gamma_n$.  First we shift these functions following $\Gamma_n$ and
define new functions on $[-4,0] \times \R^2$ by
\[ v_n(t,x) := \begin{cases}
\theta_+(t, x + \Gamma_n(t))^2, & x + \Gamma_n(t) \in \Omega_n, \\
0, & x + \Gamma_n(t) \notin \Omega_n.
\end{cases} \]
Each $v_n$ is supported on $|x| \leq 3$, and
\begin{equation} \label{definition of v_n} v_n(t,x) = \paren{\theta_n(t, x + \Gamma_n(t)) - \psi(x)}_+^2 \end{equation}
whenever the right hand side is defined.
\vskip.3cm
Note that
\[ \del_t v_n(t,x) = \del_t \theta_+^2(t,x+\Gamma_n) + \dot{\Gamma}_n \cdot \grad \theta_+^2(t,x+\Gamma_n). \]

For $C$ independent of $n$, we know from \eqref{theta3 compact in space} that
\[ \norm{ v_n }_{L^2(-4,0; H^{1/2}(\R^2)} \leq C \]
and from \eqref{theta3 compact in time} that
\[ \norm{ \del_t v_n }_{\mathcal{M}(-4,0; C^{-2}(\Omega))} \leq C \]
where $\mathcal{M}$ is the space of Radon measures with total-variation norm and $C^{-2}(\Omega)$ is the dual of $C^2(\Omega)$.  

Therefore, by the Aubin-Lions Lemma, the set $\{v_n\}_n$ is compactly embedded in\\ $L^2([-4,0]\times\R^2)$.  Up to a subsequence, there is a function $v \in L^2([-4,0]\times\R^2)$ such that
\[ v_n \xrightarrow{L^2} v. \]
\vskip.3cm

We know that $v \in L^\infty$, $\supp(v) \subseteq [-4,0]\times B_3(0)$, and $v \in L^2(H^{1/2})$ because these properties hold uniformly on $v_n$.  

By \eqref{ddt theta bounded}
\begin{equation} \label{ddt v bounded} \ddt \int_{\R^2} v_n \,dx = \ddt \int_{\Omega_n} \theta_+^2 \,dx \leq C \end{equation}
so the same must be true of $v$, for $\ddt$ interpreted in the sense of distributions.  

By \eqref{mass assumption above}, \eqref{mass assumption between}, and \eqref{mass assumption below} applied to $v_n$ (recalling the relation \eqref{definition of v_n}), we conclude that
\begin{equation} \label{mass assumptions v} \begin{cases} \begin{aligned}
\abs{\{v \geq 1\} \cap [-2,0]\times B_2(0)} &\geq \delta_0/4, \\
\abs{\{0 < v < [1-\psi]^2\} \cap [-4,0]\times B_4(0)} &\leq 0, \\
\abs{\{v \leq 0\} \cap [-4,0]\times B_4(0)} &\geq 2|B_4|. \end{aligned}
\end{cases} \end{equation}


For any $(t,x) \in [-4,0]\times B_4(0)$, either $v(t,x) \geq [1 - \psi(x)]^2$ or else $v(t,x) = 0$.  In fact, since $\norm{v(t,\cdot)}_{H^{1/2}} < \infty$ for almost every $t$ and $H^{1/2}$ does not contain functions with jump discontinuities, the function $v$ is either identically 0 or else $\geq [1-\psi(x)]^2$ at each $t$.  

Thus $\int v(t,x) \,dx$ is either 0 or else $\geq \int [1-\psi(x)]^2 \,dx > 0$ at each $t$.  By \eqref{ddt v bounded} and \eqref{mass assumptions v}, $v$ must be identically zero for all $t > -2$ but also must be non-zero for some $t > -2$, which is a contradiction.  

Our assumption that the sequence $\theta_n$ exists must have been false, and the proposition must be true.  

\end{proof}

\vskip1cm
\section{A Decrease in Oscillation} \label{sec:harnack}

We combine the two De Giorgi lemmas (Propositions \ref{thm:DG1} and \ref{thm:DG2}) to produce an oscillation lemma.  This result is similar to the weak Harnack inequality for harmonic functions.  As in the previous section, all of the following results must be independent of the size of $\Omega$, and any assumptions made on $\Omega$ must be scaling invariant.  

\begin{lemma}[Oscillation Lemma] \label{thm:oscillation general}
Let $\Ccalib$, $\Comega$, and $\Cgamma$, be positive constants. Then there exists a constant $k_0>0$ such that the following holds:

Let $\Omega \subseteq \R^2$ be a bounded open set with $C^{2,\beta}$ boundary for some $\beta \in (0,1)$, and let $\Gamma:[-5,0] \to \R^2$ be Lipschitz.  Assume that on $\Omega$ the functions $K_{1/4}$ and $K_1$ (defined in \eqref{Caff Stinga representation}) satisfy the relation
\[ K_{1/4}(x,y) \leq \Comega |x-y|^{3/4} K_1(x,y) \qquad \forall x\neq y \in \Omega. \]

Let $\theta$, $\ulow$, $\uhigh$, and $\Gammal$ solve \eqref{eq:main linear brokendown} on $[-5,0]\times\Omega$, and satisfy $\norm{\Lambda^{-1/4} \uhigh}_{L^\infty([-5,0]\times\Omega)} \leq 6 \Ccalib$ and $\norm{\grad \ulow}_{L^\infty([-5,0]\times\Omega)} \leq 2\Ccalib$.  Denote $\gamma := \Gammal - \Gamma$ and assume $\norm{\dot{\gamma}}_{L^\infty([-5,0])} \leq \Cgamma$ and $\gamma(0) = 0$.  
\vskip.3cm

Suppose that for all $t \in [-5,0]$ and any $x \in \Omega$
\begin{equation} \label{oscillation boundedness} \theta(t,x) \leq 2 + 2^{-k_0} \paren{|x-\Gamma(t)|^{1/4}-2^{1/4}}_+, \end{equation}
and that
\[ \abs{\{\theta \leq 0\} \cap [-4,0]\times B_4(\Gamma)} \geq 2|B_4|. \]

Then for all $t \in [-1,0]$, $x \in \Omega \cap B_1(\Gamma)$ we have
\[ \theta(t,x) \leq 2 - 2^{-k_0}. \]
\end{lemma}

\begin{proof}
Let $\mu$ and $\delta_0$ as in Proposition \ref{thm:DG2}, and take $k_0$ large enough that $(k_0-1) \mu > 4 |B_4|$.  

Consider the sequence of functions,
\[ \theta_k(t,x) := 2 + 2^k (\theta(t,x) - 2). \]
That is, $\theta_0 = \theta$ and as $k$ increases, we scale vertically by a factor of 2 while keeping height 2 as a fixed point.  Note that since $\theta$ satisfies \eqref{oscillation boundedness}, each $\theta_k$ for $k \leq k_0$ and $(t,x) \in [-5,0] \times \Omega$ satisfies
\[ \theta_k(t,x) \leq 2 + \paren{|x-\Gamma(t)|^{1/4}-2^{1/4}}_+. \]
This is precisely the assumption in Proposition \ref{thm:DG2}.  

Note also that
\begin{equation} \label{oscillation amount below} \abs{\{\theta_k \leq 0\} \cap [-4,0]\times B_4(\Gamma)} \end{equation}
is an increasing function of $k$, and hence is greater than $2|B_4|$ for all $k$.  

Assume, for means of contradiction, that
\begin{equation} \label{oscillation amount above} \abs{\{1 \leq \theta_k \} \cap [-2,0]\times B_2(\Gamma)} \geq \delta_0/4 \end{equation}
for $k = k_0-1$.  Since this quantity is decreasing in $k$, it must then exceed $\delta_0/4$ for all $ k < k_0$ as well.  

Applying Proposition \ref{thm:DG2} to each $\theta_k$, we conclude that 
\[ \abs{\{0 < \theta_k < 1\} \cap [-4,0]\times B_4(\Gamma)} \geq \mu. \]
In particular, this means that the quantity \eqref{oscillation amount below} increases by atleast $\mu$ every time $k$ increases by 1. By choice of $k_0$ and the fact that quantity \eqref{oscillation amount below} is trivially bounded by $4|B_4|$, we obtain a contradiciton.  Therefore, the assumption \eqref{oscillation amount above} must fail for $k = k_0-1$.  
\vskip.3cm

Therefore $\theta_{k_0}$ must satisfy the assumptions of Proposition \ref{thm:DG1}.  In particular, we conclude that
\[ \theta_{k_0}(t,x) \leq 1 \qquad \forall t \in [-1,0], x \in \Omega \cap B_1(\Gamma). \]

For the original function $\theta$, this means that
\[ \theta(t,x) \leq 2 - 2^{-k_0} \qquad \forall t \in [-1,0], x \in \Omega \cap B_1(\Gamma). \]
\end{proof}

By assuming that $\theta$ is small near $x = \Gamma(t)$, we have shown that the oscillation of $\theta$ is decreased in a smaller neighborhood of $\Gamma(t)$.  However, our goal is to control the oscillation near $x = \Gammal(t))$.  Therefore we will prove the following proposition:

\begin{proposition}[Oscillation Lemma with shift] \label{thm:oscillation shifted}
Let $\Ccalib$, $\Comega$, and $\Cgamma$, be positive constants, and let $k_0$ be as in Lemma~\ref{thm:oscillation general}.  Then there exists a constant $\lambda > 0$ such that the following holds:

Let $\Omega \subseteq \R^2$ be a bounded open set with $C^{2,\beta}$ boundary for some $\beta \in (0,1)$, and let $\Gamma:[-5,0] \to \R^2$ be Lipschitz.  Assume that on $\Omega$ the functions $K_{1/4}$ and $K_1$ (defined in \eqref{Caff Stinga representation}) satisfy the relation
\[ K_{1/4}(x,y) \leq \Comega |x-y|^{3/4} K_1(x,y) \qquad \forall x\neq y \in \Omega. \]

Let $\theta$, $\ulow$, $\uhigh$, and $\Gammal$ solve \eqref{eq:main linear brokendown} on $[-5,0]\times\Omega$, and satisfy $\norm{\Lambda^{-1/4} \uhigh}_{L^\infty([-5,0]\times\Omega)} \leq 6 \Ccalib$ and $\norm{\grad \ulow}_{L^\infty([-5,0]\times\Omega)} \leq 2\Ccalib$.  Denote $\gamma := \Gammal - \Gamma$ and assume $\norm{\dot{\gamma}}_{L^\infty([-5,0])} \leq \Cgamma$ and $\gamma(0) = 0$.  
\vskip.3cm

Suppose that for all $t \in [-5,0]$ and any $x \in \Omega$
\begin{equation} \label{theta bounded everywhere} |\theta(t,x)| \leq 2 + 2^{-k_0} \paren{|x-\Gamma(t)|^{1/4}-2^{1/4}}_+ \end{equation}
and that
\[ \abs{\{\theta \leq 0\} \cap [-4,0]\times B_4(\Gamma)} \geq 2|B_4|. \]

Then for any $\eps \in (0,1/5]$ such that
\begin{equation} \label{Cgamma and eps for harnack} 5 \Cgamma \leq \eps\n - 3 \end{equation}
we have
\[ \abs{\frac{2}{2-\lambda} \bracket{\theta(\eps t, \eps x) + \lambda}} \leq 2 + 2^{-k_0} \paren{|x-\eps\n\Gammal(\eps t)|^{1/4}-2^{1/4}}_+. \]
for all $t \in [-5,0]$ and $x$ such that $\eps x \in \Omega$.  
\end{proposition}

The idea of the proof is to consider a small enough time interval that $\Gamma(t)$ is very close to $\Gammal(t)$.  This is possible because $\Gammal - \gamma$ is uniformly Lipschitz by assumption.  

If, in this proposition, we only wished to show the existence of some $\eps = \eps(k_0,\Cgamma)$ satisfying the proposition's conclusion, then a simpler non-constructive proof would suffice.  However, in Section~\ref{sec:holder} we will apply this proposition with parameters $k_0$ and $\Cgamma$ depending on $\eps$.  To avoid circularity, we must prove the result for all $\eps$ satisfying \eqref{Cgamma and eps for harnack}.

\begin{proof}
Let $\bar{\lambda}>0$ and $\alpha>1$ be the universal constants defined in Lemma~\ref{thm:technical scaling of barrier}.  
Take $\lambda>0$ such that
\begin{equation} \label{lambda smallness assumptions} 2\lambda \leq 2^{-k_0}, \qquad (2+\lambda)(\frac{2}{2-\lambda}) \leq 2 + 2^{-k_0} \bar{\lambda}, \qquad \frac{2}{2-\lambda} \leq \alpha. \end{equation}

Denote 
\[ \bar{\theta}(t,x) := \frac{2}{2-\lambda} \bracket{\theta(\eps t, \eps x) + \lambda} \]
defined for $t \in [-5/\eps,0]$ and 
\[ x \in \Omega_\eps := \{x\in \R^2: \eps x \in \Omega\} \]
and denote
\[ \phi(x) := \paren{|x|^{1/4} - 2^{1/4}}_+. \]
\vskip.3cm

We proved in Lemma \ref{thm:oscillation general} that $\theta(t,x) \leq 2 - 2^{-k_0}$ for $t \in [-1,0]$ and $x \in \Omega\cap B_1(\Gamma)$.  On this same set, $\theta(t,x) \geq -2$ by assumption.  
By the definition of $\bar{\theta}$ and by \eqref{lambda smallness assumptions}, for all $t \in [-1/\eps, 0]$ and $x \in \Omega \cap B_{1/\eps}(\eps\n \Gamma(\eps t))$ we have therefore
\begin{equation}\label{bar theta bounded basin} \begin{cases}
\bar{\theta}(t,x) &\leq \frac{2}{2-\lambda} \bracket{2-2^{-k_0}+\lambda} \leq \frac{2}{2-\lambda} \bracket{2-\lambda} = 2. \\
\bar{\theta}(t,x) &\geq \frac{2}{2-\lambda} \bracket{-2+\lambda} = -2.
\end{cases} \end{equation}

Similarly, the bound \eqref{theta bounded everywhere} on $\theta$ becomes the equivalent bounds on $\bar{\theta}$, for all $(t,x) \in [-5/\eps,0] \times \Omega_\eps$
\begin{equation} \bar{\theta}(t,x) \leq \frac{2}{2-\lambda} \bracket{2 + 2^{-k_0} \phi(|\eps x - \Gamma(\eps t)|) + \lambda} \label{bar theta bounded above everywhere} \end{equation}
and
\begin{equation} \bar{\theta}(t,x) \geq \frac{2}{2-\lambda} \bracket{- 2 - 2^{-k_0} \phi(|\eps x - \Gamma(\eps t)|) + \lambda}. \label{bar theta bounded below everywhere} \end{equation}
\vskip.3cm

Let $t \in [-5,0]$ and $x \in \Omega_\eps$, and define 
\[ y := x - \eps\n \Gamma(\eps t). \]

From \eqref{bar theta bounded above everywhere} and the assumptions \eqref{lambda smallness assumptions}, we can bound
\begin{align*} 
\bar{\theta}(t,x) &\leq \frac{2}{2-\lambda} \bracket{2 + \lambda + 2^{-k_0} \phi(\eps |y|)}
\\ &\leq 2 + 2^{-k_0} \bar{\lambda} + 2^{-k_0} \alpha \phi(\eps |y|)
\\ &= 2 + 2^{-k_0} \bracket{\bar{\lambda} + \alpha \phi(\eps |y|)}.
\end{align*}

From \eqref{bar theta bounded below everywhere} and the assumptions \eqref{lambda smallness assumptions}, we can bound
\begin{align*}
-\bar{\theta}(t,x) &\leq \frac{2}{2-\lambda} \bracket{2 -\lambda + 2^{-k_0} \phi(\eps |y|)}
\\ &\leq 2 + 2^{-k_0} \alpha \phi(\eps |y|)
\\ &\leq 2 + 2^{-k_0} \bracket{\bar{\lambda} + \alpha \phi(\eps |y|)}.
\end{align*}

Therefore
\begin{equation} \label{bar theta bounded everywhere} \abs{\bar{\theta}(t,x)} \leq 2 + 2^{-k_0} \bracket{\bar{\lambda} + \alpha \phi(\eps |y|)}. \end{equation}
\vskip.3cm

If $|y| \leq \eps\n$ then from \eqref{bar theta bounded basin} we have
\[ \abs{\bar{\theta}(t,x)} \leq 2 \leq 2 + 2^{-k_0} \phi(x - \eps\n \Gamma(\eps t) - \eps\n \gamma(\eps t)) \]
which is our desired result.  Therefore assume without loss of generality that $|y|\geq \eps\n$.  In this case we can apply Lemma \ref{thm:technical scaling of barrier} which states that, since $\eps<1/2$ and $\eps|y| \geq 1$, it is a property of $\phi$, $\alpha$, and $\bar{\lambda}$ that
\[ 2 + 2^{-k_0} \bracket{\bar{\lambda} + \alpha \phi(\eps |y|)} \leq 2 + 2^{-k_0} \bracket{\phi(|y|-\eps\n + 3)}. \]

For $t \in [-5,0]$, we have by assumption \eqref{Cgamma and eps for harnack}
\[ |y|-\eps\n+3 \leq |y| - 5 \Cgamma \leq |y - \eps\n\gamma(\eps t)|. \]

The estimate \eqref{bar theta bounded everywhere} becomes
\[ \abs{\bar{\theta}(t,x)} \leq 2 + 2^{-k_0} \phi(|x - \eps\n\Gamma(\eps t) - \eps\n\gamma(\eps t)|). \]

This concludes the proof.  
\end{proof}

\vskip1cm
\section{H\"{o}lder Continuity} \label{sec:holder}

In this section we shall prove the main theorem, Theorem~\ref{thm:main continuity}.  We will demonstrate H\"{o}lder continuity by iteratively applying Proposition~\ref{thm:oscillation shifted} and rescaling.  

We begin with a lemma to describe the scaling properties of \eqref{eq:main linear}.  

\begin{lemma}[Scaling] \label{thm:scaling}
Let $\Omega \subseteq \R^2$ be a bounded open set with $C^{2,\beta}$ boundary for some $\beta \in (0,1)$, and let $j_0 \in \Z$ and $\eps > 0$ be constants.  

Suppose that $\theta:[-T,0] \times \Omega \to \R$ and $u:[-T,0]\times \Omega \to \R^2$ are a suitable solution to \eqref{eq:main linear} and $u$ is calibrated by a sequence $(u_j)_{j \geq j_0}$ with constant $\Ccalib$ and center $N$.  

Suppose that on $\Omega$ the functions $K_{1/4}$ and $K_1$ (defined in \eqref{Caff Stinga representation}) satisfy the relation
\[ K_{1/4}(x,y) \leq \Comega |x-y|^{3/4} K_1(x,y) \qquad \forall x\neq y \in \Omega. \]
\vskip.3cm

Then
\[ \bar{\theta}(t,x) := \theta(\eps t, \eps x) \]
and
\[ \bar{u}(t,x) := \sum_{j=j_0}^\infty \bar{u}_j(t, x), \qquad \bar{u}_j(t,x) := u_j(\eps t, \eps x) \]
are also a suitable solution to \eqref{eq:main linear} on $[-T/\eps, 0]\times \Omega_\eps$ where $\Omega_\eps = \{x \in \R^2: \eps x \in \Omega\}$.  

Moreover, $(\bar{u}_j)_{j \geq j_0}$ is calibrated with the same constant $\Ccalib$ but with center $N - \log_2(\eps)$, and the relation
\begin{equation} \label{scaling 1/4 to 1 property} \bar{K}_{1/4}(x,y) \leq \Comega |x-y|^{3/4} \bar{K}_1(x,y) \qquad \forall x\neq y \in \Omega_\eps \end{equation}
holds.  

\end{lemma}

\begin{proof}
Denote by $\bar{\Lambda}$ the square root of the Laplacian with Dirichlet boundary conditions on $\Omega_\eps$.  One can calculate (see e.g. \cite{CaSt} Section 2.4) that for $(t,x) \in [-T/\eps, 0]\times \Omega_\eps$
\[ \Lambda \theta(\eps t,\eps x) = \eps \bar{\Lambda} \bar{\theta}(t,x). \]

Similarly, in the Caffarelli-Stinga representation from Proposition~\ref{thm:Caff Stinga representation} the operator $\bar{\Lambda}^s$ will have kernel
\[ \bar{K}_s(x,y) = \eps^{s-2} K_s(\eps x,\eps y). \]

From these facts it is clear that the scaled functions satisfy \eqref{eq:main linear} and \eqref{scaling 1/4 to 1 property}.  
\vskip.3cm

To show that $(\bar{u}_j)_{j \in \Z}$ is calibrated, we must translate the three bounds on $u_j$ to corresponding bounds on $\bar{u}_j$.  Each of the calculations are similar, so we show only one:
\[ \norm{\grad \bar{u}_j}_\infty = \eps \norm{\grad u_j}_\infty \leq 2^{\log_2(\eps)} 2^j 2^{-N} \Ccalib = 2^j 2^{-(N-\log_2(\eps))} \Ccalib. \]

\end{proof}
\vskip.3cm

The next lemma demonstrates H\"{o}lder continuity of suitable solutions.  The proof method is to consruct a sequence of rescaled functions all of which, by induction, satisfy the assumptions of Proposition~\ref{thm:oscillation shifted}.  We will assume that the velocity $u$ is the Riesz transform of an $L^\infty$ function $\Theta$, which will in practice typically be $\theta$ itself, up to scaling and translation.  

\begin{lemma}[Continuity of suitable solutions] \label{thm:continuity for suitable}
There exists a universal constant $C$ such that the following holds:

Let $\Omega \subseteq \R^2$ be an open, bounded domain with $C^{2,\beta}$ boundary, $\beta \in (0,1)$.  Let $\Theta \in L^\infty([-5,0]\times\Omega)$.  Then there exists a constant $\alpha \in (0,1)$ depending on $\Omega$ and $\norm{\Theta}_{L^\infty}$ such that the following holds:
\vskip.3cm

Let $\theta:[-5,0] \times \Omega \to \R$ and $u:[-5,0]\times\Omega \to \R^2$ be a suitable solution to \eqref{eq:main linear}.  Assume that $\norm{\theta}_{L^\infty([-5,0]\times\Omega)} \leq 2$ and that $u = \grad^\perp \Lambda^{-1} \Theta$.  

Then for any point $P \in \bar{\Omega}$, $\theta$ is H\"{o}lder continuous at $(0,P)$ and
\[ \sup_{(t,x) \in [-5,0]\times\Omega} \frac{|\theta(t,x) - \theta(0,P)|}{(|t|^2 + |x-P|^2)^{\alpha/2}} \leq C. \]
\end{lemma}

\begin{proof}
By relabelling our coordinate system, we can assume without loss of generality that $P = 0$ is the origin in $\R^2$.  

From Proposition~\ref{thm:u is calibrated}, we know that 
\[ u = \grad^\perp \Lambda^{-1} \Theta = \sum_{j=j_0}^\infty u_j \]
for a sequence $(u_j)_{j\geq j_0}$ of divergence-free functions calibrated with some constant $\Ccalib = \Ccalib(\Omega, \norm{\Theta}_{L^\infty})$ and center 0.  Assume without loss of generality that $j_0 < 0$.

Choose a constant $0 < \eps < 1/5$ such that
\begin{equation}\label{eps is small enough for Cgamma} 
5 \max\paren{ - \Ccalib \log_2(\eps) e^{10\eps\Ccalib}, (1-j_0) \Ccalib} \leq \eps\n - 3.
\end{equation}

For integers $k \geq 0$ consider the domains
\[ \Omega_k := \{x \in \R^2: \eps^k x \in \Omega\}. \]
If $K_s^k$ are the kernels defined in Proposition~\ref{thm:Caff Stinga representation} corresponding to the operators $\Lambda^s$ on $\Omega_k$, then by Proposition~\ref{thm:Caff Stinga representation} and Lemma~\ref{thm:scaling} the relation
\[ K^k_{1/4}(x,y) \leq \Comega |x-y|^{3/4} K^k_1(x,y) \qquad \forall x \neq y \in \Omega_k \]
holds for some constant $\Comega$ independent of $k$.  

For notational convenience, denote
\[ \sum_k = \sum_{j > - k \log_2(\eps)}, \qquad \sum^k = \sum_{j \leq -k \log_2(\eps)} \]
and define the following functions on $[-5,0]\times \Omega_k$:
\begin{align*} 
\ulowth{k}(t,x) &:= \sum^k u_j(\eps^k t, \eps^k x), \\
\uhighth{k}(t,x) &:= \sum_k u_j(\eps^k t, \eps^k x).  
\end{align*}
By Lemmas \ref{thm:scaling} we know the sequence $(u_j(\eps^k \cdot, \eps^k \cdot))_j$ is calibrated with constant $\Ccalib$ and center $-k\log_2(\eps)$, and hence by \ref{thm:calibration is good} we know that, independently of $k$,
\[ \norm{\Lambda^{-1/4} \uhighth{k}}_{L^\infty([-5,0]\times \Omega_k)} \leq 6 \Ccalib \]
and
\[ \norm{\grad \ulowth{k}}_{L^\infty([-5,0]\times \Omega_k)} \leq 2 \Ccalib. \]
Each $\ulowth{k}$ is a finite sum of $L^\infty$ functions, hence $L^\infty$ itself, though not uniformly in $k$.  
\vskip.3cm

Define $\Gamma_k, \gamma_k: [-5,0] \to \R^2$ by the following recursive formulae and ODEs:
\begin{align*}
\Gamma_0(t) &:= 0, & t \in [-5,0], \\
\gamma_k(0) &:= 0, & k \geq 0, \\
\dot{\gamma}_k(t) &:= \ulowth{k}(t, \Gamma_k(t) + \gamma_k(t)) - \dot{\Gamma}_k(t), & k \geq 0, t \in [-5,0) \\
\Gamma_k(t) &:= \eps\n \gamma_{k-1}(\eps t) + \eps^{-2} \gamma_{k-2}(\eps^2 t) + \cdots + \eps^{-k} \gamma_0(\eps^k t), & k \geq 1, t \in [-5,0].
\end{align*}
Since each $\ulowth{k}$ is $L^\infty$ in space-time and Lipschitz in space, these $\gamma_k$ exist by a version of the Cauchy-Lipschitz theorem.  For example, Theorem 3.7 of Bahouri, Chemin, and Danchin \cite{BaChDa} proves existence and uniqueness in our case.  In particular, since $\ulowth{k}$ is a vector field which is tangential to the boundary of $\Omega_k$ and has unique flows, the Lagrangian path
\[ \Gammal^k(t) := \Gamma_k(t) + \gamma(k) \]
for $\ulowth{k}$ must remain inside $\bar{\Omega}_k$ for all time and so our expressions remain well-defined.  

The quantity $\gamma_k$ here corresponds to the frequency packets $u_j$ with $-(k-1) \log_2(\eps) < j \leq -k\log_2(\eps)$.  These frequencies are included in the definition of $\ulowth{k}$ but not the definition of $\ulowth{k-1}$ (they would instead be included in $\uhighth{k-1}$).  
\vskip.3cm

By construction, for $k \geq 0$ we have $\Gamma_{k+1}(t) = \eps\n \gamma_k(\eps t) + \eps\n \Gamma_k(\eps t)$.  Therefore
\begin{align*} 
\dot{\Gamma}_{k+1}(t) &= \del_t \bracket{\eps\n \gamma_k(\eps t) + \eps\n \Gamma_k(\eps t)}
\\ &= \dot{\gamma}_k(\eps t) + \dot{\Gamma}_k(\eps t)
\\ &= \ulowth{k}(\eps t, \gamma_k(\eps t) + \Gamma_k(\eps t))
\\ &= \ulowth{k}(\eps t, \eps \Gamma_{k+1}(t)).  
\end{align*}

With this in hand, we can bound the size of $\gamma_k$.  Namely, for $k \geq 1$, 
\begin{align*}
\dot{\gamma}_k(t) &= \ulowth{k}(t, \Gamma_k(t) + \gamma_k(t)) - \dot{\Gamma}_k(t)
\\ &= \ulowth{k}(t, \Gamma_k(t) + \gamma_k(t)) - \ulowth{k-1}(\eps t, \eps \Gamma_k(t))
\\ &= \sum^k u_j(\eps^k t, \eps^k \Gamma_k(t) + \eps^k \gamma_k(t)) - \sum^{k-1} u_j(\eps^k t, \eps^k \Gamma_k(t))
\\ &= \sum^{k-1} \bracket{u_j(\eps^k t, \eps^k \Gamma_k(t)+\eps^k \gamma_k(t)) - u_j(\eps^k t, \eps^k \Gamma_k(t))} + \sum_{k-1}^k u_j(\eps^k t, \eps^k \ldots)
\\ &= \bracket{\ulowth{k-1}\big(\eps t, \eps \Gamma_k(t)+\eps \gamma_k(t)\big) - \ulowth{k-1}(\eps t, \eps \Gamma_k(t))} + \sum_{k-1}^k u_j(\eps^k t, \eps^k \ldots).
\end{align*}
The function $x \mapsto \ulowth{k-1}(\eps t, x)$ is Lipschitz, with Lipschitz constant less than $2 \Ccalib$.  Moreover, each $u_j$ has $\norm{u_j}_\infty \leq \Ccalib$.  Thus from the above calculation we can bound
\begin{equation} \label{ODE for dot gamma} |\dot{\gamma}_k(t)| \leq 2 \Ccalib \eps |\gamma_k(t)| - \Ccalib \log_2(\eps). \end{equation}
Applying Gronwall's inequality, we find that for $t \in [-5,0]$
\[ |\gamma_k(t)| \leq \frac{-\log_2(\eps)}{2 \eps} \paren{ e^{10 \eps \Ccalib} - 1}. \]
Plugging this estimate back into \eqref{ODE for dot gamma},
\[ |\dot{\gamma}_k(t)| \leq -\Ccalib \log_2(\eps) e^{10\eps \Ccalib} \qquad \forall k \geq 1. \]

Trivially $|\dot{\gamma}_0| \leq (1-j_0) \Ccalib$, so if we define
\[ \Cgamma = \max\paren{ - \Ccalib \log_2(\eps) e^{10\eps\Ccalib}, (1-j_0) \Ccalib} \]
then for all $k \geq 0$ and $t \in [-5,0]$
\[ |\dot{\gamma}_k(t)| \leq \Cgamma. \]
Moreover, the assumption \eqref{Cgamma and eps for harnack} then follows from \eqref{eps is small enough for Cgamma}.  
\vskip.3cm

Define
\[ \theta_0(t,x) := \theta(t,x) \]
and for each $k \geq 0$, if $|\{\theta_k \leq 0\} \cap [-4,0]\times B_4(\Gamma_k(t))| \geq 2|B_4|$ then set
\[ \theta_{k+1}(t,x) := \frac{2}{2-\lambda} \bracket{\theta_k(\eps t, \eps x) + \lambda}. \]
Otherwise, set
\[ \theta_{k+1}(t,x) := \frac{1}{1-\lambda} \bracket{\theta_k(\eps t, \eps x) - \lambda}. \]

From Lemma \ref{thm:scaling}, we know that $\theta_k$ and the calibrated function $\sum_{j \geq j_0} u_j(\eps^k \cdot, \eps^k \cdot)$ solve \eqref{eq:main linear}.  By construction, $\theta_k$, $\ulowth{k}$, $\uhighth{k}$, and $\Gammal^k$ solve \eqref{eq:main linear brokendown}. 
\vskip.3cm

Since $|\theta_0|\leq 2$ by assumption, we know in particular that 
\begin{equation}\label{thetak below the barrier}
|\theta_k| \leq 2 + 2^{-k_0} \paren{|x-\Gamma_k(t)|^{1/4} - 2^{1/4}}_+
\end{equation}
holds for $k=0$.  

If \eqref{thetak below the barrier} holds for $k$, then at least one of $\theta_k$ or $-\theta_k$ (depending on whether $|\{\theta_k \leq 0\} \cap [-4,0]\times B_4(\Gamma_k(t))|$ is more or less than $2|B_4|$) will satisfy the assumptions of Proposition~\ref{thm:oscillation shifted}.  In either case, we conclude that $\theta_{k+1}$ satisfies \eqref{thetak below the barrier}.  By induction, this bound holds for all $\theta_k$.  
\vskip.3cm

Each $\theta_k$ is between $-2$ and 2 on $[-5,0]\times B_2(\Gamma_k)$.  But recall that each $\Gamma_k$ is Lipschitz with constant $k \Cgamma$.  Thus $|\Gamma_k(t)|\leq 1$ for $t \in [-(k \Cgamma)\n, 0]$.  On that time interval, 
\[ \abs{\theta_k(t,x)} \leq 2 \qquad \forall x \in B_1(0). \]

We conclude that
\[ \abs{ \sup_{[-\eps^k (k \Cgamma)\n, 0] \times B_{\eps^k}(0)} \theta(t,x) - \inf_{[-\eps^k (k \Cgamma)\n, 0] \times B_{\eps^k}(0)} \theta(t,x) } \leq 4 \paren{\frac{2}{2-\lambda}}^{-k}. \]

In particular, for some positive constant $C$ such that
\[ \eps^{C k} \leq (k \Cgamma)\n \qquad \forall k \geq 0, \]
we can say that
\[ |t|^2 + |x|^2 \leq \eps^{(1+C)k} \]	
implies that $(t,x) \in [-\eps^k (k \Cgamma)\n, 0] \times B_{\eps^k}(0)$ which in turn implies that
\[ \abs{\theta(t,x) - \theta(0,0)} \leq  4 \paren{\frac{2}{2-\lambda}}^{-k}. \]

In other words,
\begin{align*} 
\abs{\theta(t,x) - \theta(0,0)} &\leq 4 \paren{\frac{2}{2-\lambda}}^{ -\frac{1}{1+C} \log_\eps(|t|^2 - |x|^2)  + 1} 
\\ &= 4 \paren{\frac{2}{2-\lambda}} \exp\bracket{\ln\paren{\frac{2}{2-\lambda}} \frac{\ln(|t|^2 + |x|^2)}{-(1+C)\ln(\eps)}}
\\ &= \frac{8}{2-\lambda} (|t|^2 + |x|^2)^{-\frac{\ln(2) - \ln(2-\lambda)}{(1+C)\ln(\eps)}}
\\ &\leq 8 (|t|^2 + |x|^2)^{-\frac{\ln(2) - \ln(2-\lambda)}{(1+C)\ln(\eps)}}
\end{align*}

\end{proof}

We are now able to prove the main result, Theorem~\ref{thm:main continuity}.  

\begin{proof}[Proof of Theorem \ref{thm:main continuity}]
Recall that $\Omega$, $S$, $k$, and $\theta_0$ are given.  

In Proposition~\ref{thm:existence} we construct global-in-time solutions to \eqref{eq:main nonlinear}.  By construction, there is a universal constant $C_1$ so $\norm{\theta(t,\cdot)}_{L^\infty(\Omega)} \leq C_1 t\n \norm{\theta_0}_{L^2(\Omega)}$.  

Consider a point $(t_0, x_0)$ with $t_0 > S$.  Consider arbitrary constants $\lambda, \mu \in (0,1]$ and note that
\[ \tilde{\theta}(t,x) := \lambda \theta(t_0 + \mu t, \mu x), \qquad \tilde{u}(t,x) := u(t_0 + \mu t, \mu x) \]
is a suitable solution to \eqref{eq:main linear} on $[-t_0/\mu,\infty) \times \tilde{\Omega}$ where $\tilde{\Omega} := \{x\in \R^2: \mu x \in \Omega \}$.  

If $S + \mu (-5) = \frac{S}{2}$, or equivalently if $\mu = S/10$, then then we have 
\[ \norm{\tilde{\theta}}_{L^\infty([-5,0]\times\tilde{\Omega})} \leq \lambda 2 C_1 \frac{k}{S}. \]
Take $\lambda = S / (C_1 k)$.

On $[-5,0]\times\tilde{\Omega}$ we have $\tilde{\theta}$ and $\tilde{u}$ a suitable solution to \eqref{eq:main linear} satisfying $\norm{\tilde{\theta}}_{L^\infty} \leq 2$ and $\tilde{u} = \grad^\perp \Lambda^{-1} \Theta$ with $\norm{\Theta}_{L^\infty} \leq 2 C_1 k/S$.  Therefore we can apply Lemma~\ref{thm:continuity for suitable} to $\tilde{\theta}$, $\tilde{u}$ and find that $\tilde{\theta}$ satisfies, for $\alpha = \alpha(k,S)$ and $C$ universal,
\[ \sup_{(t,x) \in [-5,0]\times\tilde{\Omega}} \frac{|\tilde{\theta}(t,x) - \tilde{\theta}(t_0,x_0)|}{\paren{|t-t_0|^2 + |x-x_0|^2}^{\alpha/2}} \leq C. \]

For the original unscaled $\theta$, we have
\[ \sup_{(t,x) \in [S/2,t_0]\times\Omega} \frac{|\theta(t,x) - \theta(t_0,x_0)|}{\paren{|t-t_0|^2 + |x-x_0|^2}^{\alpha/2}} \leq C \lambda\n \mu^{-\alpha} \leq C (\lambda\mu)\n = C \frac{10 C_1}{S^2} k. \]

\end{proof}


\vskip1cm
\appendix
\section{Technical Lemmas} \label{sec:technical}

In this appendix we state and prove a few technical lemmas.  

\begin{lemma}[De Giorgi Iteration Argument] \label{thm:DG1 skeleton}
For any constant $\bar{C} \geq 0$, there exists a $\delta>0$ such that the following holds:

Let $\Omega \subseteq \R^2$ be a bounded open set with $C^{2,\beta}$ boundary for some $\beta \in (0,1)$.  
Let $f \in L^2([-2,0]\times\Omega)$ be a function with the property that for any positive constant $a$
\begin{equation} \label{DG energy ddt} \ddt \int (f-a)_+^2 + \int \abs{\Lambda^{1/2} (f-a)_+}^2 \leq \bar{C} \paren{ \int \indic{f \geq a} + \int (f-a)_+ + \int (f-a)_+^2 }. \end{equation}

Then
\[ \int_{-2}^0 \int (f-0)_+^2 \,dxdt \leq \delta \]
implies that
\[ f(t,x) \leq 1 \qquad \forall t\in[-1,0], x \in \Omega. \]
\end{lemma}

\begin{proof}
Consider for $k \in \N$ the constants $t_k := -1 - 2^{-k}$ (so that $t_0 = -2$ and $t_\infty = -1$), and functions
\[ f_k := (f - 1 + 2^{-k})_+ \]
(so that $f_0 = (f)_+$ and $f_\infty = (f-1)_+$).  

Define
\[ \E_k := \int_{t_k}^0 \int_\Omega f_k^2 \,dxdt. \]

When $f_{k+1}>0$, then in particular $f_k \geq 2^{-k-1}$.  Thus for any finite $p$, there exists a constant $C$ so
\begin{equation} \label{non-linearization} \indic{f_{k+1}>0} \leq C^k f_k^p. \end{equation}
\vskip.3cm

Let $k\geq 0$ and define $\eta:[-2,0] \to \R$ a continuous function
\[ \eta(t) := \begin{cases}
0 & t \leq t_k \\
2^{k+1} (t-t_k) & t_k \leq t \leq t_{k+1} \\
1 & t_{k+1} \leq t.
\end{cases} \]

Let $s \in (t_{k+1},0)$.  Multiplying the inequality \eqref{DG energy ddt} with cutoff $a_k$ by $\eta(t)$ and integrating in time from $-2$ to $s$, then integrating by parts, we obtain
\[ \int f_k(s,x)^2\,dx - 2^{k+1} \int\displaylimits_{t_k}^{t_{k+1}} \int f_k(t,x)^2 \,dxdt + \int\displaylimits_{t_{k+1}}^s \HDint{1/2}{f_k} \,dxdt \leq \bar{C} \paren{ \int\displaylimits_{t_k}^0 \int \indic{f_k>0} + f_k + f_k^2 \,dxdt } \]
By taking the supremum over all $s \in (t_{k+1},0)$, we obtain
\begin{equation} \label{DG energy sup}
\sup_{[t_{k+1},0]} \int f_k^2 \,dx + \int_{t_{k+1}}^0 \HDint{1/2}{f_k}\,dxdt \leq C \paren{ 2^{k+1} \int_{t_k}^0 \int f_k^2\,dxdt + \int_{t_k}^0 \int \indic{f_k > 0} + f_k \,dxdt } 
\end{equation}
\vskip.3cm

From Proposition~\ref{thm:hadamard 3 lines} and Sobolev embedding, 
\[ \int_{t_{k+1}}^0 \paren{ \int f_k^4 \,dx }^{1/2} \,dt \leq C \int_{t_{k+1}}^0 \HDint{1/2}{f_k} \,dxdt. \]
Therefore by the Riesz-Thorin interpolation theorem,
\[ \int_{t_{k+1}}^0 \int f_k^3 \,dxdt \leq C \paren{ \sup_{[t_{k+1},0]} \int f_k^2 \,dx + \int_{t_{k+1}}^0 \HDint{1/2}{f_k} }^{3/2}. \]
This estimate, along with \eqref{DG energy sup} and \eqref{non-linearization}, and the fact that $t_{k-1} < t_k$ and $f_{k-1} \geq f_k$, tell us that
\[ \int_{t_{k+1}}^0 \int f_k^3 \,dxdt \leq C^k \E_{k-1}^{3/2}. \]

Now we can estimate, using again \eqref{non-linearization} and the fact $f_k \geq f_{k+1}$,
\[ \E_{k+1} \leq C^k \int_{t_{k+1}}^0 \int f_k^3 \,dxdt \leq C^k \E_{k-1}^{3/2}. \]
\vskip.3cm

This nonlinear recursive inequality $\E_{k+1} \leq C^k \E_{k-1}^{3/2}$, by a standard fact about nonlinear recursions (see \cite{DG} or \cite{Va.dg}), tells us that there exists a constant $\delta$ depending only on $C$ (which in turn depends only on the constant $\bar{C}$ in \eqref{DG energy ddt})
\[ \E_0 \leq \delta \textrm{ implies } \lim_{k \to \infty} \E_k = 0. \]

By assumption
\[ \E_0 = \int_{-2}^0 \int (f)_+ \leq \delta. \]
Therefore $\E_k \to 0$ and, by the dominated convergence theorem,
\[ \int_{-1}^0 \int (f-1)_+ \,dxdt = 0. \]

The result follows.  
\end{proof}


\begin{lemma} \label{thm:interpolation C1 to holder}
Let $\alpha \in (0,1)$.  There exists a constant $C = C(\alpha)$ such that, for any set $\Omega$ and any $f \in C^{0,1}(\Omega)$,
\[ \bracket{f}_\alpha \leq C \norm{f}_\infty^{1-\alpha} \norm{\grad f}_\infty^\alpha. \]
\end{lemma}

\begin{proof}
This simple lemma is a straightforward calculation:
\begin{align*} 
\sup_{x,y \in \Omega} \frac{|f(x)-f(y)|}{|x-y|^\alpha} &= \sup |f(x)-f(y)|^{1-\alpha} \paren{\frac{|f(x)-f(y)|}{|x-y|}}^\alpha 
\\ &\leq \paren{2 \norm{f}_\infty}^{1-\alpha} \paren{ \sup \frac{|f(x)-f(y)|}{|x-y|} }^\alpha
\\ &\leq C \norm{f}_\infty^{1-\alpha} \norm{\grad f}_\infty^\alpha.
\end{align*}
\end{proof}

\begin{lemma} \label{thm:interpolation holder to C1}
Let $\alpha \in (0,1)$ and $\Omega$ a set that satisfies the cone condition.  There exist constants $C = C(\alpha, \Omega)$ and $\ell = \ell(\Omega)$ such that, for any $f \in C^{1,\alpha}(\Omega)$
\[ \norm{\grad f}_\infty \leq C \paren{ \delta\n \norm{f}_\infty  + \delta^\alpha \bracket{\grad f}_\alpha }\]
for all $\delta < \ell$.  
\end{lemma}

The idea of the proof is to average $\grad f$ along an interval of length $\delta$ with endpoint $x$.  The magnitude of the average will be small, since $f \in L^\infty$, and the average will differ not very much from $\grad f(x)$ since $\grad f \in C^{1,\alpha}$.  

\begin{proof}
Since $\Omega$ satisfies the cone condition, there exist positive constants $\ell$ and $a<1$ such that, at each point $x \in \bar{\Omega}$, there exist two unit vectors $e_1$ and $e_2$ such that $|e_1\cdot e_2| \leq a$ and $x + \tau e_i \in \Omega$ for $i=1,2$ and $0 < \tau \leq \ell$.  In other words, $\Omega$ contains rays at each point that extend for length $\ell$, end at $x$, and are non-parallel with angle at least $\cos\n(a)$.  

Consider the directional derivative $\del_i f$ of $f$ along the direction $e_i$, and observe that for any $0 < \delta \leq \ell$,
\begin{equation} \label{average bounded by Linfty} \abs{\int_0^\delta \del_i f(x + \tau e_i) \,d\tau} = \abs{f(x+\delta e_i) - f(x)} \leq 2 \norm{f}_\infty. \end{equation}

On the other hand, $\del_i f$ is continous so, for any $\tau \in (0,\ell]$,
\[ \abs{\del_i f(x) - \del_i f(x+\tau e_i)} \leq \bracket{\grad f}_\alpha \tau^\alpha. \]
From this, we obtain that
\[ \int_0^\delta \del_i f(x + \tau e_i) \,d\tau \leq \int_0^\delta \paren{\del_i f(x) + \bracket{\grad f}_\alpha \tau^\alpha } \,d\tau = \delta \del_i f(x) + \bracket{\grad f}_\alpha \frac{\delta^{1+\alpha}}{1+\alpha} \]
and a similar bound holds from below.  Thus
\[ \abs{ \delta \del_i f(x) - \int_0^\delta \del_i f(x + \tau e_i) \,d\tau} \leq \bracket{\grad f}_\alpha \frac{\delta^{1+\alpha}}{1+\alpha}. \]

Combining this bound with \eqref{average bounded by Linfty}, we obtain
\[ \abs{\del_i f(x)} \leq \frac{2}{\delta} \norm{f}_\infty + \frac{\delta^\alpha}{1+\alpha} \bracket{\grad f}_\alpha. \]

This bound is independent of $x$ and of $i=1,2$.  Since $e_1 \cdot e_2 \leq a$ by assumption, by a little linear algebra we can bound $\grad f$ in terms of the $\del_i f$ and obtain that, for all $\delta \in (0,\ell]$,
\[ \norm{\grad f}_\infty \leq \frac{C}{1-a^2} \paren{ \delta\n \norm{f}_\infty + \delta^\alpha \bracket{\grad f}_\alpha }. \]

\end{proof}

\begin{lemma} \label{thm:technical scaling of barrier}
There exist constants $\bar{\lambda} > 0$ and $\alpha > 1$ such that, for any $0 < \eps \leq 1/2$ and any $z \geq 1$
\[ \paren{|\eps\n (z - 1) + 3|^{1/4} - 2^{1/4}}_+ - \alpha \paren{|z|^{1/4} - 2^{1/4}}_+ \geq \bar{\lambda}. \]
\end{lemma}

\begin{proof}
For $z$ fixed, this function is increasing as $\eps$ decreases, so it will suffice to show the lemma when $\eps = 1/2$, that is to show
\[ f_\alpha(z) := \paren{|2 z + 1|^{1/4} - 2^{1/4}}_+ - \alpha \paren{|z|^{1/4} - 2^{1/4}}_+ \geq \bar{\lambda} \]
for all $z \geq 1$.  Note that $f_\alpha(z) \geq f_\beta(z)$ if $\alpha < \beta$.  

For $z \geq 2$, 
\[ f_\alpha(z) = (2z+1)^{1/4} - 2^{1/4} - \alpha z^{1/4} + \alpha 2^{1/4} = z^{1/4} \paren{(2 + 1/z)^{1/4} - \alpha} + (\alpha-1)2^{1/4}. \]
For any $\alpha < 2^{1/4}$, clearly $f_\alpha(z)$ tends to $\infty$ as $z$ increases. Therefore there exist $N$ and $\alpha_0 > 1$ such that 
\[ f_\alpha(z) \geq 1 \qquad \forall z \geq N, \alpha\leq \alpha_0. \]
\vskip.3cm

We can decompose $f_\alpha(z) = g_1(z) - (\alpha - 1) g_2(z)$ where
\begin{align*} 
g_1(z) &:= \paren{|2 z + 1|^{1/4} - 2^{1/4}}_+ - \paren{|z|^{1/4} - 2^{1/4}}_+, \\
g_2(z) &:= \paren{|z|^{1/4} - 2^{1/4}}_+. 
\end{align*}
Note that $g_1$, $g_2$ are both continuous, and $g_1(z)$ is strictly positive for $z \geq 1$.  Therefore we can take $\alpha \in (1,\alpha_0]$ small enough that
\[ \alpha - 1 < \frac{\inf_{[1,N]} g_1}{\sup_{[1,N]} g_2}. \]

For this $\alpha$, $f_\alpha(z)$ is strictly positive on the compact interval $[1,N]$, and $f_\alpha(z) \geq 1$ on $[N,\infty)$.  Therefore $f_\alpha(z)$ has a positive lower bound $\bar{\lambda}$ for all $z \geq 1$.  
\end{proof}


%
%
%


\bibliographystyle{alpha}
\bibliography{SQG-references}

\end{document}